\documentclass{amsart}
\usepackage{amsmath,amssymb,amsthm}
\usepackage{hyperref}

\numberwithin{equation}{section}
\newtheorem{theorem}{Theorem}[section]
\newtheorem{lemma}[theorem]{Lemma}
\newtheorem{proposition}[theorem]{Proposition}
\newtheorem{definition}[theorem]{Definition}
\newtheorem{corollary}[theorem]{Corollary}

\newcommand{\orcid}[1]{\href{https://orcid.org/#1}{\texttt{ORCID: #1}}}
\newcommand{\cpar}{c_{\text{par}}}
\newcommand{\cE}{\mathcal{E}}
\newcommand{\C}{\mathbb{C}}
\newcommand{\al}{\alpha}
\newcommand{\be}{\beta}
\newcommand{\ga}{\gamma}
\newcommand{\Ga}{\Gamma}
\newcommand{\de}{\delta}
\newcommand{\e}{\varepsilon}
\newcommand{\la}{\lambda}
\newcommand{\om}{\omega}
\newcommand{\cG}{\mathcal {G}}
\newcommand{\wt}{\widetilde}
\newcommand{\Id}{{\bf{1}}}
\newcommand{\bA}{{\bf A}}
\newcommand{\Kc}{K_\circ}
\newcommand{\cB}{{\mathcal B}}
\newcommand{\A}{\mathbb A}
\newcommand{\R}{\mathbb{R}}
\newcommand{\N}{\mathbb{N}}
\newcommand{\bfone}{\mathbf{1}}
\newcommand{\rap}{{\rm RapDec}}
\newcommand{\dist}{{\rm dist}}
\newcommand{\BL}{\textup{BL}}
\newcommand{\ang}{\measuredangle}
\newcommand{\supp}{\operatorname{supp}}

\begin{document}

\title[Oscillatory Integral Operators with Cinematic Curvature]{%
On Oscillatory Integral Operators Satisfying\\
the Cinematic Curvature Condition}

\date{}

\author{Xiangyu Wang} \address{Xiangyu Wang \\ Department of Mathematics, University of Illinois Urbana-Champaign\\ \newline \orcid{0009-0003-5983-6961}} \email{xw70@illinois.edu}

\begin{abstract}

We establish improved $L^r\to L^p$ estimates for oscillatory integral
operators satisfying the positive cinematic curvature condition. In ambient
dimension $n\geq4$, the output range is $p>2+8/(3n-4)$, with the full
Knapp-admissible range $r'\leq (n-2)p/n$. The proof combines a
transverse covering $L^2$ estimate, Wolff's two-ends reduction, and refined
decoupling. In dimension $n=3$ we retain the range $p>4$. We also construct new
counterexamples based on Kakeya compression.

\end{abstract}


\maketitle

\section{Introduction}
Let $n \ge 3$. We write $x=(x',x_n)\in\R^{n-1}\times\R$ and $\xi=(\xi',\xi_{n-1})\in\R^{n-2}\times\R$. For $y\in\R^m$ and $r>0$, we denote by $B^m_r(y)\subset\R^m$ the ball of radius $r$ centered at $y$, and by $B^m_r$ a ball of radius $r$ in $\R^m$ whose center is not specified. Let
\[\A:=\{ (\xi',\xi_{n-1})\in\R^{n-1}: 1\le \xi_{n-1}\le 2, |\xi'|\le \xi_{n-1} \}. \]
Let $a \in C_c^{\infty}(\mathbb R^n \times \mathbb R^{n-1})$ be a non-negative function supported in $B^n_1(0) \times \mathbb{A}$, and let $\phi\colon B^n_1(0) \times \mathbb{A} \to \mathbb R$ be a phase function satisfying the following three conditions, which together are called the \emph{cinematic curvature condition}.
\begin{enumerate}
\item[\textbf{(H1)}] \emph{Homogeneity.} The function $\phi(x;\xi)$ is homogeneous of degree $1$ in $\xi$ and smooth away from $\xi = 0$.
\item[\textbf{(H2)}] \emph{Non-degeneracy.} We have $\operatorname{rank}\partial^2_{x'\xi}\phi(x',x_n;\xi)=n-1$ for all $(x',x_n;\xi)\in\supp a$ with $\xi \ne 0$.
\item[\textbf{(H3)}] \emph{Positive definiteness.} Define the Gauss map $G\colon \supp a\rightarrow S^{n-1}$ by $G(x;\xi):=G_0(x;\xi)/|G_0(x;\xi)|$, where
    \[G_0(x;\xi):=\bigwedge_{j=1}^{n-1}\partial_{\xi_j}\partial_{x}\phi(x;\xi). \]
Then, for all $(x,\xi_0)\in\supp a$, the matrix
\[
\partial^2_{\xi\xi}\langle \partial_x \phi(x,\xi),G(x,\xi_0) \rangle|_{\xi=\xi_0}
\]
has rank $n-2$ and $n-2$ positive eigenvalues.
\end{enumerate}
For $\la \geq 1$ let $a^{\lambda}(x; \xi) := a(x/\lambda; \xi)$ and $\phi^{\lambda}(x;\xi) :=\lambda\phi(x/\lambda;\xi)$, and define the operator $T^{\lambda}$ by
\[
T^{\lambda}f(x) := \int e^{-2\pi i\phi^{\lambda}(x;\xi)}a^{\lambda}(x;\xi)f(\xi)d\xi
\]
for integrable functions $f \colon \R^{n-1} \to \C$. We call $T^{\lambda}$ an \emph{oscillatory integral operator satisfying the cinematic curvature condition}, and we write $G^{\lambda}(x;\xi):=G(x/\lambda;\xi)$ for the rescaled Gauss map. The basic example is the phase $\phi(x;\xi)=\langle x',\xi\rangle+x_n|\xi'|^2/(2\xi_{n-1})$, for which $T^\lambda$ is, up to the amplitude, the Fourier extension operator of the cone $\{(\xi,|\xi'|^2/(2\xi_{n-1})):\xi\in\A\}$; in suitable coordinates, this is a piece of the light cone.

Define
\begin{equation}\label{eq:new-endpoint-pair}
 p_n:=\begin{cases}
 4,&n=3,\\[2pt]
 \displaystyle 2+\frac8{3n-4},&n\geq4,
 \end{cases}
 \qquad
 r_n:=\begin{cases}
 4,&n=3,\\[2pt]
 \displaystyle 2+\frac{4}{3n-8},&n\geq4.
 \end{cases}
\end{equation}
These exponents satisfy $2<r_n\leq p_n$ and $r_n'=(n-2)p_n/n$, where $r'$ denotes the exponent conjugate to $r$. Throughout the paper, $n$ is the dimension of the space of the variable $x$; in the notation of \cite{GW} for Fourier integral operators, the corresponding spatial dimension is $n-1$.

Our main result is the following.
\begin{theorem} \label{main theorem}
Let $T^{\lambda}$ be an oscillatory integral operator satisfying the cinematic curvature condition. Then the estimate
\begin{equation}  \label{linear estimate}
    \|T^{\lambda}f\|_{L^p(\mathbb{R}^n)} \lesssim_{\phi, a} \|f\|_{L^{r}(\mathbb{R}^{n-1})}
\end{equation}
holds uniformly in $\lambda \geq 1$ whenever
\begin{equation} \label{main range}
p>p_n
\end{equation}
and
\begin{equation} \label{rrange}
    \frac{r}{r-1} \leq \frac{n-2}{n} p.
\end{equation}
\end{theorem}

The condition \eqref{rrange}, that is, $r'\leq(n-2)p/n$, is necessary for \eqref{linear estimate}, as the Knapp example shows.

Schippa \cite{S} proved the estimate \eqref{linear estimate} in the diagonal case $r=p$ for
\[
p >
\begin{cases}
4, & \text{if } n=3,\\[4pt]
2+\dfrac{8}{3n-3}, & \text{if } n>3 \text{ is odd},\\[4pt]
2+\dfrac{8}{3n-4}, & \text{if } n>3 \text{ is even},
\end{cases}
\]
by means of polynomial partitioning; this extends to variable coefficients the estimates of Ou and Wang \cite{OW} for the cone. For $n=3$ and for even $n$, this is the range $p>p_n$ of Theorem~\ref{main theorem}; for odd $n\geq5$, the exponent $p_n$ is slightly larger than the exponent of \cite{S}. On the other hand, Theorem~\ref{main theorem} covers all exponents $r$ in the range \eqref{rrange}. For such exponents less was known in high dimensions: when $n\geq 8$, the mixed-norm estimates discussed in \cite{S}, which are obtained by extending the argument of Ou and Wang \cite[Theorem~2]{OW}, require
\[
p>2+\frac{4}{n+1}.
\]
Theorem~\ref{main theorem} improves this threshold for every $n\geq 8$.

\subsection*{Necessary conditions}
For the cone, the estimate \eqref{linear estimate} holds for all $p>2(n-1)/(n-2)$ and all $r$ with $r'<(n-2)p/n$; this is the cone restriction conjecture, which follows from the recently announced proof of the local smoothing conjecture \cite{OpenAI2026LocalSmoothing}. For variable coefficients, the optimal exponent $p$ is larger in high dimensions. The reason is \emph{Kakeya compression}: for suitable phases, the wave packets of $T^\lambda f$ can be arranged so that they all lie in a thin neighborhood of a set of low dimension; see, for instance, \cite{BG,GHI}. Using examples of this kind, Schippa \cite[Proposition~2.1]{S} showed that, for $n\geq6$, the estimate \eqref{linear estimate} with $r=p$ can hold for all phases only if
\[
p\geq
\begin{cases}
\displaystyle 2+\frac8{3n-1},&n\ \text{odd},\\[4pt]
\displaystyle 2+\frac8{3n-2},&n\ \text{even},
\end{cases}
\]
and he left open whether the exponent in his diagonal estimate is sharp \cite[Remark~2.2]{S}. The functions in these examples are bounded and supported in sets of bounded measure, so the same condition is necessary for every $r$.

In Section~\ref{section-necessary} we construct phases for which the wave packets concentrate near a plane of dimension $\lfloor(n+1)/2\rfloor$, and we obtain the following necessary conditions (Proposition~\ref{prop:necessary}). Let $n\geq5$. If the estimate \eqref{linear estimate} holds for every oscillatory integral operator satisfying the cinematic curvature condition, even with an additional factor $\lambda^\varepsilon$ on the right-hand side, then
\begin{equation}\label{eq:intro-necessary}
p\geq2+\frac{8}{3n-3}\quad\text{if $n$ is odd},
\qquad\qquad
\frac{6n}{p}+\frac4r\leq3n-2\quad\text{if $n$ is even}.
\end{equation}
For odd $n$, the condition \eqref{eq:intro-necessary} does not involve $r$, and its exponent is the one in the diagonal estimate of \cite{S}. Hence the exponent of that estimate is sharp, up to the endpoint, in every odd dimension; for $n=3$ and $n=5$ this follows already from the necessary condition for the cone. For even $n$ and $r=p$, the condition \eqref{eq:intro-necessary} is the condition $p\geq2+8/(3n-2)$ of \cite{S}, but for $r<p$ it is more restrictive. On the line $r'=(n-2)p/n$, which is the boundary of the range \eqref{rrange}, it becomes
\[
p\geq2+\frac{8(n-3)}{3(n-2)^2},
\]
and this is a stronger condition than $p\geq2+8/(3n-2)$ when $n\geq8$.

The exponent $p_n$ in Theorem~\ref{main theorem} is close to the lower bounds given by these necessary conditions. For odd $n\geq5$, the exponent $p_n$ exceeds the lower bound in \eqref{eq:intro-necessary} by $8/((3n-3)(3n-4))$. For even $n\geq6$ and $r'=(n-2)p/n$, it exceeds the lower bound in the last display by $8n/(3(n-2)^2(3n-4))$. Both quantities are $O(n^{-2})$. For instance, when $n=8$ the necessary condition on the line $r'=(n-2)p/n$ is $p\geq64/27$, while $p_8=12/5$.

\subsection*{Strategy of the proof}
By an $\varepsilon$-removal argument and bilinear interpolation (see \cite{OW,S}), Theorem~\ref{main theorem} follows from the corresponding local estimate with an $R^\varepsilon$ loss.
\begin{proposition} \label{epsilon main theorem}
Let $T^{\lambda}$ be an oscillatory integral operator satisfying the cinematic curvature condition. Then, for every $\varepsilon > 0$, the estimate
\[
\|T^{\lambda}f\|_{L^p(B^n_R)} \lesssim_{\phi,a,\varepsilon} R^\varepsilon \|f\|_{L^{r}(\mathbb{R}^{n-1})}
\]
holds uniformly for $1 \leq R \leq \lambda$ whenever
\begin{equation}\label{range of result}
p\geq p_n
\end{equation}
and $r$ satisfies \eqref{rrange}.
\end{proposition}

The proof of Proposition~\ref{epsilon main theorem} follows the scheme used by Wang and Wu \cite{WW} for the restriction problem and by Gan, He, Li, and Wu \cite{GW} for local smoothing estimates: Wolff's two-ends reduction, a broad--narrow argument, and refined decoupling. To these we add the transverse covering $L^2$ estimate of \cite{WangDensity}. The two-ends reduction provides a number $\mu$ that controls, from above and from below, the number of wave packets passing through each ball of the set on which $T^\lambda f$ is to be estimated, after the wave packets with directions in a few caps of width $K^{-1}$ have been discarded; here $K$ is a small power of $R$. In the refined decoupling estimate, at the exponent $2n/(n-2)$, the number $\mu$ appears as a loss; in the covering estimate, at the exponent $2$, it appears as a gain. When $n\geq4$ and these two estimates are interpolated with the standard $L^2$ estimate, the powers of $\mu$, of $R$, and of the number of wave packets cancel exactly at the exponents $(p_n,r_n)$; when $n=3$, a simpler $L^2$ estimate takes the place of the covering estimate.

\subsection*{Structure of the article}
In Section~\ref{section-necessary} we prove the necessary conditions. Section~\ref{section-preliminary} contains the preliminary reductions and the tools that we need, including the wave packet decomposition, the broad norm, and refined decoupling. Section~\ref{section-algorithm} is devoted to the two-ends reduction, and Section~\ref{concludetheproof} completes the proof of Proposition~\ref{epsilon main theorem}. Finally, in Section~\ref{epsilon removal} we prove an $\varepsilon$-removal lemma for oscillatory integral operators satisfying the cinematic curvature condition, which gives Theorem~\ref{main theorem}.

\subsection*{Notation} We write $\#E$ for the cardinality of a finite set $E$. If $\mathcal{E}$ is a family of subsets of $\mathbb{R}^n$, then $\cup \mathcal{E}$ denotes $\bigcup_{E \in \mathcal{E}} E$. For $A,B \ge 0$, we write $A \lesssim B$ if $A \le C B$ for a constant $C$ that does not depend on $\lambda$, $R$, or the function $f$; further dependence of $C$ on parameters is indicated by subscripts. We write $A \sim B$ if $A \lesssim B$ and $B \lesssim A$. For $R>1$, we write $A \lessapprox B$ if $A \le C_\eta R^\eta B$ for every $\eta>0$. Finally, $\mathrm{RapDec}(R)$ denotes a quantity that is bounded by $C_N R^{-N}$ for all $R\geq1$, where $N$ is a number that may be taken as large as we need.

\subsection*{Choice of parameters} Throughout the paper, $0 < \varepsilon \ll 1$ and $1 \leq R \leq \lambda$, and we set
\[
K = R^{\varepsilon^{50}},\qquad K_{\circ} = R^{\varepsilon^{100}},\qquad \kappa = R^{\varepsilon^{500}},\qquad
\delta = \varepsilon^{1000}.
\]
Thus
\[
R^{\delta} \ll \kappa \ll K_{\circ} \ll K \ll R \leq \lambda .
\]

\subsection*{Acknowledgments} The author is deeply grateful to Prof.~Xiaochun Li for suggesting this research topic, for many insightful discussions, and for carefully reading earlier versions of this manuscript; his guidance was invaluable throughout this work. The author also thanks Prof.~Shukun Wu for several helpful conversations and for comments that clarified some technical aspects of the paper.

Claude were used to assist with language editing. The counterexamples in Section~\ref{section-necessary}, which give the necessary conditions of Proposition~\ref{prop:necessary}, were constructed by ChatGPT.

\section{Necessary conditions}\label{section-necessary}
In this section we prove the necessary conditions stated in the introduction. They come from Kakeya compression: we construct phase functions for which a large family of wave packets, pointing in different directions, is contained in a thin neighborhood of a plane of low dimension. This section is independent of the rest of the paper, and the notation introduced here will not be used in the later sections.

\begin{proposition}\label{prop:necessary}
Let $n\geq5$, $2\leq p<\infty$, and $1\leq r\leq\infty$. Suppose that, for every oscillatory integral operator $T^\lambda$ satisfying the cinematic curvature condition and for every $\varepsilon>0$, the estimate
\begin{equation}\label{eq:necessary-hypothesis}
\|T^\lambda f\|_{L^p(B^n_\lambda(0))}\lesssim_{\phi,a,\varepsilon}\lambda^{\varepsilon}\|f\|_{L^r(\R^{n-1})}
\end{equation}
holds uniformly in $\lambda\geq1$. If $n$ is odd, then
\begin{equation}\label{eq:necessary-odd}
p\geq 2+\frac{8}{3n-3}.
\end{equation}
If $n$ is even, then
\begin{equation}\label{eq:necessary-even}
\frac{6n}{p}+\frac{4}{r}\leq 3n-2.
\end{equation}
\end{proposition}

Since $T^\lambda f$ is supported in $B^n_\lambda(0)$, the hypothesis \eqref{eq:necessary-hypothesis} is implied by the estimate \eqref{linear estimate} of Theorem~\ref{main theorem}, and also by the estimate of Proposition~\ref{epsilon main theorem} with $R=\lambda$. Before we prove the proposition, we compare it with the known results.

In odd dimensions, the condition \eqref{eq:necessary-odd} does not depend on $r$; indeed, the functions $f$ in our example are bounded by $1$ and supported in $\A$. Its exponent is the one in the diagonal estimate of Schippa \cite{S}, so that estimate is sharp up to the endpoint in every odd dimension $n\geq5$. This answers, in odd dimensions, the question left open in \cite[Remark~2.2]{S}. For $n=5$ the condition \eqref{eq:necessary-odd} reads $p\geq8/3$, which is, up to the endpoint, the condition $p>2(n-1)/(n-2)$ that is necessary already for the cone; for odd $n\geq7$ it is stronger than that condition.

In even dimensions, the condition \eqref{eq:necessary-even} with $r=p$ reads $p\geq2+8/(3n-2)$, which is the necessary condition of \cite[Proposition~2.1]{S}. For $r<p$ it is more restrictive than that condition. Consider the line $r'=(n-2)p/n$, which is the boundary of the range \eqref{rrange}. Inserting $1/r=1-n/((n-2)p)$ into \eqref{eq:necessary-even}, we obtain
\begin{equation}\label{eq:necessary-knapp}
p\geq 2+\frac{8(n-3)}{3(n-2)^2}.
\end{equation}
Since
\[
\frac{8(n-3)}{3(n-2)^2}-\frac{8}{3n-2}=\frac{8(n-6)}{3(n-2)^2(3n-2)},
\]
this is a stronger condition than $p\geq2+8/(3n-2)$ for every even $n\geq8$; for $n=8$, for example, the lower bound for $p$ on this line rises from $26/11$ to $64/27$.

\subsection{The phase function}\label{subsec:compression-phase}
Write $n=2k+1+\ell$ with $\ell\in\{0,1\}$. Thus $k\geq2$, and $\ell=0$ if $n$ is odd, $\ell=1$ if $n$ is even. We split the variables as
\begin{gather*}
x=(v,w,u,t)\in\R^k\times\R^\ell\times\R^k\times\R,\\
\xi=(\beta,\gamma,\alpha)\in\R^k\times\R^\ell\times\R^k,
\qquad
\alpha=(\alpha',\rho)\in\R^{k-1}\times\R .
\end{gather*}
Hence $t=x_n$, $\rho=\xi_{n-1}$, and $\xi'=(\beta,\gamma,\alpha')$. When $\ell=0$, the variables $w$ and $\gamma$ are absent. For $\rho>0$ let
\[
 h(\alpha):=\frac{|\alpha'|^2}{2\rho},\qquad
 d(\alpha):=\nabla_\alpha h(\alpha)
     =\left(\frac{\alpha'}{\rho},-\frac{|\alpha'|^2}{2\rho^2}\right),
\]
and define
\begin{equation}\label{eq:compression-phase}
 \phi(x;\xi)
 :=u\cdot\alpha+v\cdot\beta+w\cdot\gamma+t h(\alpha)
 +\frac{t}{2\rho}\left(
     |\beta-\rho(u+t d(\alpha))|^2+|\gamma|^2\right).
\end{equation}
Since $h$ is homogeneous of degree $1$ and $d$ is homogeneous of degree $0$, the function $\phi$ is smooth for $\rho>0$ and homogeneous of degree $1$ in $\xi$.

\begin{lemma}\label{lem:compression-phase}
There is a constant $c_0>0$ such that the phase \eqref{eq:compression-phase} satisfies \textup{(H2)} and \textup{(H3)} at every point of $B^n_{c_0}(0)\times\A$, and such that $\|\partial^2_{x'\xi}\phi-I_{n-1}\|_{\mathrm{op}}\leq1/2$ there.
\end{lemma}
\begin{proof}
A direct computation gives
\[
 \partial_{x'}\phi(0;\xi)=\xi,\qquad
 \partial_{x_n}\phi(0;\xi)=\frac{|\alpha'|^2+|\beta|^2+|\gamma|^2}{2\rho}
 =\frac{|\xi'|^2}{2\xi_{n-1}}
\]
for all $\xi$ with $\rho>0$. Hence $\partial_x\phi(0;\xi)=\partial_x\phi_0(0;\xi)$, where
\[
\phi_0(x;\xi):=\langle x',\xi\rangle+x_n\frac{|\xi'|^2}{2\xi_{n-1}}
\]
is the phase of the cone, which was mentioned in the introduction. The conditions \textup{(H2)} and \textup{(H3)} at a point $x$ involve only the function $\xi\mapsto\partial_x\phi(x;\xi)$. At $x=0$ they therefore hold for $\phi$, because they hold for $\phi_0$. Indeed, write $h_0(\xi):=|\xi'|^2/(2\xi_{n-1})$. The matrix $\partial^2_{x'\xi}\phi_0$ is the identity, and $G_0=(-\nabla h_0,1)$, where the orientation is the usual one, for which the cone satisfies \textup{(H3)}. Hence the matrix in \textup{(H3)} is a positive multiple of the Hessian of $h_0$, which is positive semidefinite of rank $n-2$.

For every $x$, the function $\xi\mapsto\partial_x\phi(x;\xi)$ is homogeneous of degree $1$, so the matrix in \textup{(H3)} annihilates the vector $\xi_0$ and has rank at most $n-2$. Its remaining $n-2$ eigenvalues depend continuously on $(x,\xi_0)$ and are positive at $x=0$. Since $\A$ is compact, they are positive for all $(x,\xi_0)\in B^n_{c_0}(0)\times\A$ if $c_0$ is small enough. The same argument applies to the matrix $\partial^2_{x'\xi}\phi$, which is the identity at $x=0$.
\end{proof}

Let
\[
\A_0:=\{\xi\in\R^{n-1}:\ 5/4\leq|\xi|\leq3/2,\ |\xi'|\leq\xi_{n-1}/10\}\subset\A .
\]
We fix a non-negative amplitude $a\in C_c^\infty(\R^n\times\R^{n-1})$ that is supported in $B^n_{c_0}(0)\times\A$ and satisfies $a(x;\xi)=1$ for $|x|\leq c_0/2$ and $\xi\in\A_0$. By Lemma~\ref{lem:compression-phase}, the operator $T^\lambda$ with this amplitude and the phase \eqref{eq:compression-phase} is an oscillatory integral operator satisfying the cinematic curvature condition.

\subsection{Central curves in a plane}\label{subsec:compression-curves}
Let $c_1\in(0,1/20]$ and $t_0\in(0,1]$ be small constants, which will be chosen below, and let $\mathcal D$ be the set of unit vectors $\theta=(\beta_\theta,\gamma_\theta,\alpha_\theta)\in\R^{n-1}$, with $\alpha_\theta=(\alpha_\theta',\rho_\theta)$, such that
\[
\gamma_\theta=0,\qquad |\beta_\theta|+|\alpha_\theta'|\leq c_1,\qquad \rho_\theta>0 .
\]
For $\theta\in\mathcal D$ put $q_\theta:=\beta_\theta/\rho_\theta\in\R^k$, and for $|t|\leq t_0$ let $X_\theta(t)\in\R^{n-1}$ be the point with coordinates
\[
v=0,\qquad w=0,\qquad u=q_\theta-t\,d(\alpha_\theta).
\]
At the point $x=(X_\theta(t),t)$ and the frequency $\xi=\theta$, we have
\[
\beta-\rho(u+t d(\alpha))=\beta_\theta-\rho_\theta q_\theta=0
\qquad\text{and}\qquad\gamma=0 .
\]
The last term in \eqref{eq:compression-phase} is $t/(2\rho)$ times the sum of the squares of these quantities, so it vanishes to second order at $\xi=\theta$. Differentiating the remaining terms, we get $\partial_\beta\phi=v=0$, $\partial_\gamma\phi=w=0$, and $\partial_\alpha\phi=u+t d(\alpha_\theta)=q_\theta$. Therefore
\begin{equation}\label{eq:compression-center}
 \partial_\xi\phi(X_\theta(t),t;\theta)=V_\theta
 \qquad\text{for all }|t|\leq t_0,
\end{equation}
where $V_\theta\in\R^{n-1}$ denotes the vector whose $\beta$- and $\gamma$-components vanish and whose $\alpha$-component is $q_\theta$.

The identity \eqref{eq:compression-center} means that the wave packets with frequency direction $\theta$ and position $V_\theta$ are concentrated along the curve $t\mapsto(X_\theta(t),t)$; this is made precise in Lemma~\ref{lem:compression-packet} below. We call this curve the \emph{central curve} of these wave packets. (In the notation of Section~\ref{section-preliminary}, we have $X_\theta(t)=\Phi(V_\theta,t;\theta)$.) The central curves are line segments, and all of them lie in the $(k+1)$-dimensional plane
\[
\Pi:=\{(v,w,u,t)\in\R^n:\ v=0,\ w=0\}.
\]
Since $|q_\theta|+|d(\alpha_\theta)|\lesssim c_1$, we can choose $c_1$ and $t_0$ so small that
\[
|(X_\theta(t),t)|\leq c_0/4\qquad\text{for all }\theta\in\mathcal D\text{ and }|t|\leq t_0 .
\]

\subsection{Wave packets}
Let $\lambda\geq1$ be large, and let $c_2\in(0,1/4]$ be a small constant. For each $\theta\in\mathcal D$, let $\chi_\theta$ be a smooth function with $0\leq\chi_\theta\leq1$ that is supported in the set
\[
\{\xi\in\R^{n-1}:\ 5/4\leq|\xi|\leq3/2,\ |\xi/|\xi|-\theta|\leq c_2\lambda^{-1/2}\},
\]
that is, over the cap of radius $c_2\lambda^{-1/2}$ around $\theta$ in the unit sphere, and satisfies $\int\chi_\theta\sim\lambda^{-(n-2)/2}$. Then $\supp\chi_\theta\subset\A_0$. Define
\[
 f_\theta(\xi):=e^{2\pi i\lambda V_\theta\cdot\xi}\chi_\theta(\xi).
\]

\begin{lemma}\label{lem:compression-packet}
If $c_2$ is sufficiently small and $\lambda$ is sufficiently large, then for every $\theta\in\mathcal D$ there is a set $P_\theta\subset B^n_\lambda(0)$ with the following properties.
\begin{enumerate}
\item[\textup{(i)}] $|P_\theta|\sim\lambda^{n/2}$.
\item[\textup{(ii)}] $|T^\lambda f_\theta(x)|\gtrsim\lambda^{-(n-2)/2}$ for all $x\in P_\theta$.
\item[\textup{(iii)}] Every $x=(x',x_n)\in P_\theta$ satisfies $|x_n|\leq t_0\lambda$ and $|x'-\lambda X_\theta(x_n/\lambda)|\lesssim\lambda^{1/2}$.
\end{enumerate}
The implicit constants do not depend on $\theta$ or $\lambda$.
\end{lemma}
\begin{proof}
For $|t|\leq t_0$ and $X'\in\R^{n-1}$ with $|(X',t)|<c_0$, put
\[
Y:=\partial_\xi\phi(X',t;\theta)-V_\theta .
\]
By \eqref{eq:compression-center}, $Y=0$ when $X'=X_\theta(t)$. Since $\partial_\xi\phi$ is homogeneous of degree $0$ in $\xi$ and $\frac54\theta\in\A$, Lemma~\ref{lem:compression-phase} shows that the map $X'\mapsto Y$ is a diffeomorphism from the set $\{|X'-X_\theta(t)|<c_0/4\}$ onto its image, it is bi-Lipschitz with constants independent of $\theta$ and $t$, and its image contains the set $\{|Y|<c_0/8\}$. Let $c_3>0$ be a small constant, and let $P_\theta$ be the set of all points $x=\lambda(X',t)$ with $|t|\leq t_0$, $|X'-X_\theta(t)|<c_0/4$, and
\begin{equation}\label{eq:compression-plate}
 |Y\cdot\theta|\leq c_3\lambda^{-1},\qquad
 |Y-(Y\cdot\theta)\theta|\leq c_3\lambda^{-1/2}.
\end{equation}
For fixed $t$, the set of vectors $Y$ satisfying \eqref{eq:compression-plate} has measure comparable to $\lambda^{-1}\lambda^{-(n-2)/2}$, and it is contained in $\{|Y|\leq2c_3\lambda^{-1/2}\}$. Hence the corresponding set of points $X'$ has measure comparable to $\lambda^{-1}\lambda^{-(n-2)/2}$, and it lies within distance $O(\lambda^{-1/2})$ of $X_\theta(t)$. After the dilation by $\lambda$ and integration over $|x_n|\leq t_0\lambda$, this proves \textup{(i)} and \textup{(iii)}. Moreover, $|x|\leq c_0\lambda/2$ for all $x\in P_\theta$ when $\lambda$ is large, so that $P_\theta\subset B^n_\lambda(0)$.

To prove \textup{(ii)}, fix $x=\lambda(X',t)\in P_\theta$ and let
\[
F(\xi):=\phi(X',t;\xi)-V_\theta\cdot\xi .
\]
Since $a(X',t;\xi)=1$ for $\xi\in\supp\chi_\theta$, we have
\[
T^\lambda f_\theta(x)=\int e^{-2\pi i\lambda F(\xi)}\chi_\theta(\xi)\,d\xi .
\]
The function $F$ is homogeneous of degree $1$ and $\nabla F(\theta)=Y$, so Euler's identity gives $F(\theta)=Y\cdot\theta$. Let $\xi=s\omega\in\supp\chi_\theta$ with $|\omega|=1$ and $s=|\xi|$. By Taylor's formula,
\[
 |F(\xi)|=s|F(\omega)|\leq 2\bigl(|Y\cdot\theta|+|Y|\,|\omega-\theta|+C|\omega-\theta|^2\bigr),
\]
where $C$ depends only on the second derivatives of $\phi$. By \eqref{eq:compression-plate} and since $|\omega-\theta|\leq c_2\lambda^{-1/2}$, it follows that
\[
 \lambda|F(\xi)|\leq2\bigl(c_3+2c_2c_3+Cc_2^2\bigr)\leq\frac1{100},
\]
provided that $c_2$ and $c_3$ are small enough. Then the real part of $e^{-2\pi i\lambda F(\xi)}$ is at least $1/2$ on the support of $\chi_\theta$, and therefore
\[
|T^\lambda f_\theta(x)|\geq\frac12\int\chi_\theta\gtrsim\lambda^{-(n-2)/2}.\qedhere
\]
\end{proof}

\subsection{Proof of Proposition~\ref{prop:necessary}}
The set $\mathcal D$ is a cap of dimension $n-2-\ell$ in the unit sphere of $\R^{n-1}$. Let $\mathcal D_\lambda\subset\mathcal D$ be a maximal $\lambda^{-1/2}$-separated subset, so that
\[
N:=\#\mathcal D_\lambda\sim\lambda^{(n-2-\ell)/2}.
\]
Since $c_2\leq1/4$, the functions $\chi_\theta$, $\theta\in\mathcal D_\lambda$, have pairwise disjoint supports. For signs $\epsilon_\theta\in\{-1,1\}$, which will be chosen below, let
\[
 f:=\sum_{\theta\in\mathcal D_\lambda}\epsilon_\theta f_\theta .
\]
Then $|f|\leq1$, and the support of $f$ has measure $O(N\lambda^{-(n-2)/2})=O(\lambda^{-\ell/2})$. Hence
\begin{equation}\label{eq:compression-input}
\|f\|_{L^r(\R^{n-1})}\lesssim\lambda^{-\ell/(2r)}
\end{equation}
for every choice of the signs.

Let $U:=\bigcup_{\theta\in\mathcal D_\lambda}P_\theta$. By part \textup{(iii)} of Lemma~\ref{lem:compression-packet}, every point of $U$ lies within distance $O(\lambda^{1/2})$ of a point $\lambda(X_\theta(t),t)$, and these points belong to the plane $\Pi$. Thus every $(v,w,u,t)\in U$ satisfies $|v|+|w|\lesssim\lambda^{1/2}$ and $|u|+|t|\lesssim\lambda$, and consequently
\begin{equation}\label{eq:compression-volume}
 |U|\lesssim\lambda^{\frac{k+\ell}{2}}\lambda^{k+1}=\lambda^D,
 \qquad
 D:=n-\frac{k+\ell}{2}=
 \begin{cases}(3n+1)/4,&n\text{ odd},\\3n/4,&n\text{ even}.\end{cases}
\end{equation}

Now choose the signs $\epsilon_\theta$ independently and uniformly at random. By parts \textup{(i)} and \textup{(ii)} of Lemma~\ref{lem:compression-packet},
\[
 \mathbb E\|T^\lambda f\|_{L^2(U)}^2
  =\sum_{\theta\in\mathcal D_\lambda}\|T^\lambda f_\theta\|_{L^2(U)}^2
  \gtrsim N\lambda^{-(n-2)}\lambda^{n/2}\sim\lambda^{1-\ell/2}.
\]
We fix signs for which $\|T^\lambda f\|_{L^2(U)}^2$ is at least as large as its expectation. Since $p\geq2$ and $U\subset B^n_\lambda(0)$, H\"older's inequality and \eqref{eq:compression-volume} give
\[
\|T^\lambda f\|_{L^p(B^n_\lambda(0))}\geq|U|^{-(\frac12-\frac1p)}\|T^\lambda f\|_{L^2(U)}
\gtrsim\lambda^{\frac12-\frac\ell4-D(\frac12-\frac1p)}.
\]
Together with \eqref{eq:compression-input}, this yields
\begin{equation}\label{eq:compression-growth}
 \frac{\|T^\lambda f\|_{L^p(B^n_\lambda(0))}}{\|f\|_{L^r(\R^{n-1})}}
 \gtrsim\lambda^{\,\frac12-\frac\ell4+\frac{\ell}{2r}-D(\frac12-\frac1p)} .
\end{equation}
If \eqref{eq:necessary-hypothesis} holds for every $\varepsilon>0$, then the exponent on the right-hand side of \eqref{eq:compression-growth} cannot be positive. For odd $n$ this means that
\[
\frac12\leq\frac{3n+1}{4}\Bigl(\frac12-\frac1p\Bigr),
\]
which is \eqref{eq:necessary-odd}. For even $n$ it means that
\[
\frac14+\frac1{2r}\leq\frac{3n}{4}\Bigl(\frac12-\frac1p\Bigr),
\]
which is \eqref{eq:necessary-even}. This completes the proof of Proposition~\ref{prop:necessary}.

\subsection{An example with a phase that is linear in \texorpdfstring{$x'$}{x'}}
The phase \eqref{eq:compression-phase} depends on $x'$ in a nonlinear way. We now show that the same necessary conditions can be obtained from phases of the form
\begin{equation}\label{eq:affine-phase}
\phi(x;\xi)=\langle x',\xi\rangle+F(x_n;\xi),
\end{equation}
which is the type of phase considered in \cite[Section~2]{S}. As explained in \cite[Remark~2.2]{S}, for the examples given there one lacks an algebraic relation that involves the coordinate $x_{n-1}$, which is dual to the radial frequency variable. In the example below, the central curves satisfy a relation that involves all the coordinates of $x'$.

Let first $n=2k+1$ be odd, with $k\geq2$. Write $\xi=(a,b)\in\R^k\times\R^k$ and $x'=(x_a,x_b)\in\R^k\times\R^k$, and put $t=x_n$. Let $s_1=2$, $s_2=3$, $s_3=4$, and let $\nu>0$ be a small constant. For $|t|<1$ and for linearly independent vectors $a$ and $b$, define
\[
 F(t;a,b):=\nu\sum_{j=1}^3\frac{|b+s_j a|}{s_j-t},
\]
and let $\phi$ be given by \eqref{eq:affine-phase}. Then $\phi$ is homogeneous of degree $1$ in $\xi$, and $\partial^2_{x'\xi}\phi$ is the identity matrix. As in the proof of Lemma~\ref{lem:compression-phase}, the matrix in \textup{(H3)} is a positive multiple of the Hessian of
\[
\xi\mapsto\partial_tF(t;\xi)=\nu\sum_{j=1}^3\frac{|b+s_j a|}{(s_j-t)^2}.
\]
Each summand is a convex function of $\xi$, so this Hessian is positive semidefinite, and a vector $(A,B)\in\R^k\times\R^k$ lies in its kernel if and only if $B+s_jA$ is parallel to $b+s_ja$ for $j=1,2,3$, that is,
\[
(B+s_jA)\wedge(b+s_ja)=B\wedge b+s_j(A\wedge b+B\wedge a)+s_j^2\,A\wedge a=0
\qquad(j=1,2,3).
\]
A quadratic polynomial that vanishes at three points is zero, so $B\wedge b=0$, $A\wedge a=0$, and $A\wedge b+B\wedge a=0$. The first two equations give $B=c'b$ and $A=ca$ for some $c,c'\in\R$, and then the third gives $(c-c')\,a\wedge b=0$, that is, $c=c'$. Hence the kernel is spanned by the radial vector $(a,b)$, and the Hessian has rank $2k-1=n-2$ and $n-2$ positive eigenvalues. Thus $\phi$ satisfies \textup{(H1)}--\textup{(H3)} on a conic neighborhood of any point $(a_0,b_0)$ with $a_0$ and $b_0$ linearly independent. After a rotation of the variables $\xi$, together with the same rotation of $x'$, we may assume that $(a_0,b_0)$ is a positive multiple of the last coordinate vector, which is the axis of $\A$; the amplitude is then chosen as in Section~\ref{subsec:compression-phase}, with its $\xi$-support in this conic neighborhood.

Next we determine the central curves. Let $h_j(\xi):=(b+s_ja)/|b+s_ja|$. Then
\[
\nabla_aF(t;\xi)=\nu\sum_{j=1}^3\frac{s_jh_j(\xi)}{s_j-t},
\qquad
\nabla_bF(t;\xi)=\nu\sum_{j=1}^3\frac{h_j(\xi)}{s_j-t}.
\]
For a unit vector $\theta$ in the conic neighborhood, set
\[
V_\theta:=\Bigl(\nu\sum_{j=1}^3h_j(\theta),\,0\Bigr),
\qquad
Z_\theta(t):=-\nu\sum_{j=1}^3\frac{h_j(\theta)}{s_j-t},
\qquad
X_\theta(t):=\bigl(tZ_\theta(t),Z_\theta(t)\bigr).
\]
Since $1-s_j/(s_j-t)=-t/(s_j-t)$, we have $V_\theta-\nabla_\xi F(t;\theta)=X_\theta(t)$, and therefore
\[
\partial_\xi\phi(X_\theta(t),t;\theta)=V_\theta\qquad\text{for all }|t|\leq t_0<1,
\]
which is the analogue of \eqref{eq:compression-center}. The central curves are no longer straight, but all of them lie in the $(k+1)$-dimensional surface
\[
\{(x_a,x_b,t)\in\R^n:\ x_a=t\,x_b\}.
\]
Moreover, $|X_\theta(t)|\lesssim\nu$, so the curves stay in the region where the amplitude equals $1$ if $\nu$ is small.

The rest of the argument is the same as before. Lemma~\ref{lem:compression-packet} and its proof apply without change to the functions $f_\theta=e^{2\pi i\lambda V_\theta\cdot\xi}\chi_\theta$, because they use only the identity \eqref{eq:compression-center}, the homogeneity of $\phi$, and the non-degeneracy of $\partial^2_{x'\xi}\phi$. Hence every point $(x_a,x_b,x_n)$ of $U=\bigcup_\theta P_\theta$ satisfies $|x_a-(x_n/\lambda)x_b|\lesssim\lambda^{1/2}$. For fixed $(x_b,x_n)$, this confines $x_a$ to a ball of radius $O(\lambda^{1/2})$ in $\R^k$, so that $|U|\lesssim\lambda^{k+1}\lambda^{k/2}=\lambda^{(3n+1)/4}$, as in \eqref{eq:compression-volume}. The argument with random signs then gives \eqref{eq:compression-growth} with $\ell=0$, and hence \eqref{eq:necessary-odd}.

In even dimensions $n=2k+2$ we add one pair of variables $(w,\eta)\in\R\times\R$, so that the space variable is $(x_a,x_b,w,t)$ and the frequency variable is $(\xi,\eta)$, and we consider the phase
\[
x_a\cdot a+x_b\cdot b+w\eta+F(t;\xi)+\frac{t\,\eta^2}{2L(\xi)},
\]
where $L$ is a linear form with $L(a_0,b_0)>0$. At $\eta=0$, the Hessian of the $t$-derivative of this phase with respect to $(\xi,\eta)$ is block diagonal, with the blocks $\partial^2_{\xi\xi}\partial_tF$ and $1/L(\xi)$. Therefore \textup{(H3)} holds at $\eta=0$ and, by homogeneity and continuity, also for $|\eta|\leq c|\xi|$ with a small constant $c$. For directions $\theta$ with $\eta_\theta=0$, the central curves are given by $w=0$ and $(x_a,x_b)=X_\theta(t)$. Using a $\lambda^{-1/2}$-separated set of $\sim\lambda^{(n-3)/2}$ such directions, we obtain $|U|\lesssim\lambda^{k+1}\lambda^{k/2}\lambda^{1/2}=\lambda^{3n/4}$ and $\|f\|_{L^r}\lesssim\lambda^{-1/(2r)}$. This gives \eqref{eq:compression-growth} with $\ell=1$, and hence \eqref{eq:necessary-even}.

\section{Preliminaries}\label{section-preliminary}

In this section we collect the reductions and the tools that are used in the proof of Proposition~\ref{epsilon main theorem}. Recall that we write $x=(x',x_n)\in \R^n$ and $\xi=(\xi',\xi_{n-1})\in\R^{n-1}$.

\subsection{Basic reductions}\label{subsecquant}
Proposition \ref{epsilon main theorem} can be reduced to the case in which $1 \le R \le \lambda^{1-\varepsilon}$ and $f \in \mathcal{S}(\mathbb{R}^{n-1})$ with $\supp f \subset \mathbb{A}$. Moreover, after an appropriate rescaling, we may assume that $B^n_R=B^n_R(0)$ and that
\begin{equation} \label{assumpetion 1}
    \|\partial_{x,\xi}^\beta a\|_{L^\infty(B^n_1(0) \times \mathbb{A})} \lesssim_\beta 1\qquad \textrm{for all $1 \leq |\beta| \leq N$,}
\end{equation}
where $N=N_{\e,p}\in\N$ is a large integer that depends only on the dimension $n$ and on $\e$ and $p$. Hence, to prove Proposition \ref{epsilon main theorem}, it suffices to prove the following.
\begin{proposition} \label{reduction 1}
Let $T^{\lambda}$ be an oscillatory integral operator satisfying the cinematic curvature condition, and suppose that the amplitude $a$ satisfies \eqref{assumpetion 1}. Let $p$ and $r$ satisfy \eqref{range of result} and \eqref{rrange}. Then, for every $\varepsilon > 0$, the estimate
\[
\|T^{\lambda}f\|_{L^p(B^n_R(0))} \lesssim_{\phi,\varepsilon,p} R^\varepsilon \|f\|_{L^{r}(\mathbb{R}^{n-1})}
\]
holds uniformly for $f \in \mathcal{S}(\mathbb{R}^{n-1})$ with $\supp f \subset \mathbb{A}$ and $1 \leq R \leq \lambda^{1-\varepsilon}$.
\end{proposition}

We shall prove Proposition \ref{reduction 1} by induction on scales. A single operator $T^\lambda$ is not stable under the rescaling used in such an argument, but the rescaling preserves a suitable class of operators, and the induction can be carried out for the whole class. We therefore strengthen Proposition \ref{reduction 1} to an estimate that is uniform over such a class, working in the framework of \cite[Section 2.3]{BHS}.

Fix a small constant $\cpar>0$. For $\bA=(A_1,A_2,A_3)\in[1,\infty)^3$, consider the following conditions on the phase and the amplitude:

\begin{enumerate}
    \item[$(\text{H1}_{\bA})$]  $\|\partial^2_{\xi x'}\phi(x;\xi)-I_{n-1}\|_{\mathrm{op}}\le \cpar A_1$ for all $(x;\xi)\in B^n_1(0)\times \mathbb{A}$.
    \item[$(\text{H2}_{\bA})$] $\|\partial^2_{\xi'\xi'}\partial_{x_n}\phi(x;\xi)-\frac{1}{\xi_{n-1}}I_{n-2}\|_{\mathrm{op}}\le \cpar A_2$ for all $(x;\xi)\in B^n_1(0)\times \mathbb{A}$.
    \item[$(\textup{D1}_\bA)$]$\|\partial_\xi^\beta\partial_{x_k}\phi\|_{L^\infty(B_1^n(0)\times \mathbb{A})}\le \cpar A_1$ for all $1\le k\le n-1$ and $\beta\in\N_0^{n-1}$ with $2\le|\beta|\le 3$ satisfying $|\beta'|\ge 2$;

    \noindent $\|\partial_{\xi'}^{\beta'}\partial_{x_k}\phi\|_{L^\infty(B_1^n(0)\times \mathbb{A})}\le \frac{\cpar}{2n} A_1$ for all $1\le k\le n-1$ and $\beta'\in\N_0^{n-2}$ with $|\beta'|=3$.
    \item[$(\textup{D2}_\bA)$] $
        \|\partial_\xi^\beta\partial_x^\al\phi\|_{L^\infty(B_1^n(0)\times \mathbb{A})}\le \frac{\cpar}{2n}A_3 $
    for all $(\al,\beta)\in\N_0^n\times \N_0^{n-1}$ with $2\le |\al|\le 4N$ and $1\le |\beta|\le 4N+2$ satisfying $1\le |\beta|\le 4N$ or $|\beta'|\ge 2$.
    \item[$(\textup{M}_\bA)$] $\dist(\supp_x a, \R^n\setminus B_1^n)\ge \frac{1}{4A_3}$.

\end{enumerate}

A pair $(\phi,a)$ is said to be of \textit{type} $\bA$ if, in addition to \textup{(H1)} and \textup{(H2)}, it satisfies the conditions $(\textup{H1}_\bA)$, $(\textup{H2}_\bA)$, $(\textup{D1}_\bA)$, $(\textup{D2}_\bA)$, and $(\textup{M}_\bA)$. By decomposing the amplitude and applying suitable rescalings, one reduces matters to the normalized case $\bA=\mathbf{1}:=(1,1,1)$; see \cite[Section 2.5]{BHS} for details.

In the remainder of this paper we assume that the operator $T^\lambda$ is of type $\bfone$. With this normalization, Proposition \ref{reduction 1} follows from the next statement.

\begin{proposition}\label{main theorem type 1}
Let $T^{\lambda}$ be an oscillatory integral operator satisfying the cinematic curvature condition of type $\bfone$, and assume \eqref{assumpetion 1}. Let $p$ satisfy \eqref{range of result} and let $r$ satisfy \eqref{rrange}. Then, for every $\varepsilon > 0$, the estimate
\[
\|T^{\lambda}f\|_{L^p(B_R^n(0))} \lesssim_{\varepsilon,p,\phi}R^\varepsilon \|f\|_{L^{r}(\mathbb{R}^{n-1})}
\]
holds uniformly for $f \in \mathcal{S}(\mathbb{R}^{n-1})$ with $\supp f \subset \mathbb{A}$ and $1 \leq R \leq \lambda^{1-\varepsilon}$.
\end{proposition}

\subsection{Wave packet decomposition}

Fix $0<\e \ll 1$ and a scale $R$ with $1 \ll R \leq \la^{1-\e}$. In this subsection we recall the wave packet decomposition of a function $f \in \mathcal{S}(\mathbb{R}^{n-1})$ at scale $R$, following \cite[Section 3.3]{GW}.

By \textup{(H2)}, for each fixed pair $(x_n,\xi)$ the map $x' \mapsto \partial_\xi \phi(x;\xi)$ has a non-vanishing Jacobian determinant, so it is locally invertible. After a suitable rescaling, we may assume that, for all $(x_n,\xi) \in [-1,1] \times \mathbb{A}$,
\begin{equation} \label{assumption2}
\partial_\xi \phi(\cdot, x_n; \xi) \quad \text{defines a diffeomorphism on}\quad
B_1^{n-1}(0).
\end{equation}
We use the following notation for the inverse map.

\begin{definition}[{\cite[Definition 3.3]{GW}}]
    For $(x_n,\xi)\in [-1,1]\times \mathbb{A}$, let $\Phi(u,x_n;\xi)$ be the unique solution to
    \begin{equation}\label{defPhi}
        \partial_\xi\phi(\Phi(u,x_n;\xi),x_n;\xi)=u.
    \end{equation}
Here $u$ ranges over a small neighborhood of the origin.
\end{definition}

Since $\phi$ is homogeneous of degree one in $\xi$, the function $\partial_\xi \phi$ is homogeneous of degree zero in $\xi$. Consequently,
\[
\Phi(u,x_n;\xi)=\Phi(u,x_n;s\xi)\qquad\text{for all }s>0.
\]
Moreover, the inverse function theorem and the conditions
$(\textup{D1}_\bA)$ and $(\textup{D2}_\bA)$ give uniform bounds for
$\Phi$ and its derivatives. Differentiating \eqref{defPhi} with respect to $u$, we obtain
\[
\partial^2_{\xi, x'= \Phi(u,x_n;\xi)} \phi \cdot \partial_u \Phi = I_{n-1},
\]
and hence
\[
\det \partial_u \Phi(u,x_n;\xi) \gtrsim 1;
\]
that is, the matrix $\partial_u \Phi(u,x_n;\xi)$ is non-singular in a quantitative sense.

Define $\gamma^\lambda\colon \mathbb{R}^{n-1} \times \mathbb{R} \times \mathbb{A} \to \mathbb{R}^{n-1}$ by
\[
\ga^\la(u,x_n;\xi):=\la \Phi(u/\la,x_n/\la;\xi).
\]
Then
\[
\partial_\xi\phi^\la(\ga^\la(u,x_n;\xi),x_n;\xi)=u.
\]
For each fixed pair $(u,\xi)$, the map $x_n \mapsto (\gamma^\lambda(u,x_n;\xi), x_n)$, $x_n \in [-R,R]$, parametrizes a curve. The wave packets will be concentrated in anisotropic neighborhoods of such curves.

We now describe the decomposition. Cover $B^{n-2}_1(0)$ by finitely overlapping caps $\theta$ of diameter $\sim R^{-1/2}$, and let $\Theta_{R^{-1/2}} := \{\theta\}$ be the collection of these caps, which we call \emph{$R^{-1/2}$-caps}. Let $\{\psi_\theta\}$ be a smooth partition of unity subordinate to this cover. For each $\theta$, let $\omega_\theta$ denote the center of $\theta$, and define
\[
\tilde{\psi}_\theta(\xi',\xi_{n-1}) :=
\begin{cases}
\psi_\theta(\xi'/\xi_{n-1}),
& \text{if } (\xi',\xi_{n-1}) \in \mathbb{A}, \\[4pt]
0,
& \text{otherwise}.
\end{cases}
\]
With $f_\theta := \tilde{\psi}_\theta f$ we have
\[ f = \sum_{\theta\in\Theta_{R^{-1/2}}}f_\theta.\]
Let $\tilde{\theta} \subset \mathbb{R}^{n-1}$ be a box of dimensions $1 \times R^{-1/2} \times \dots\times R^{-1/2}$ such that
\[\{(\xi',\xi_{n-1}):\frac{\xi'}{\xi_{n-1}} \in \theta, 1 \leq \xi_{n-1} \leq 2\} \subset \tilde{\theta}.\]
For a constant $C\geq1$, we denote by $C\tilde{\theta}$ a box of dimensions $1 \times CR^{-1/2} \times \dots\times CR^{-1/2}$ such that
\[\{(\xi',\xi_{n-1}):\frac{\xi'}{\xi_{n-1}} \in C\theta, 0.9 \leq \xi_{n-1} \leq 2.1\} \subset C\tilde{\theta}.\]

We shall also use caps of larger diameter. For $R^{-1/2}\leq s\leq1$, we denote by $\Theta_s$ a finitely overlapping cover of $B^{n-2}_1(0)$ by caps of diameter $\sim s$, which we call \emph{$s$-caps}. For $\tau\in\Theta_s$, the center $\omega_\tau$, the cutoff $\tilde\psi_\tau$, the function $f_\tau:=\tilde\psi_\tau f$, and the boxes $\tilde\tau$ and $C\tilde\tau$ are defined as above, with $R^{-1/2}$ replaced by $s$.

Before we decompose in physical space, we introduce some notation.
\begin{definition}[{\cite[Definition 3.4]{GW}}] \label{dual-box}
Let $U\subset \R^m$ be a box of dimensions $a_1\times \dots\times a_m$. Define $U^*$ to be the {\bf dual box} of $U$ centered at the origin of dimensions $a_1^{-1}\times \dots\times a_m^{-1}$, and with edges parallel to the corresponding edges of $U$.
\end{definition}
\begin{definition}[{\cite[Remark 3.6]{GW}}] \label{dilation}
For a slab $U\subset \R^{n-1}$ of dimensions $Rs^2\times Rs\times \dots\times Rs$ with $s\in[R^{-1/2},1]$ and a constant $C>0$, define $CU$, the non-isotropic $C$-dilation of $U$, to be a slab of dimensions $C^2Rs^2\times CRs\times \dots\times CRs$ with the same center as $U$.

Two slabs $U_1, U_2\subset \R^{n-1}$ are said to be \textbf{comparable} if
    \[\frac{1}{C}U_1\subset U_2\subset C U_1.\]
Here $C=C_\phi$ is a constant that may vary from line to line, but depends only on $\phi$.

\end{definition}

Fix $\theta \in \Theta_{R^{-1/2}}$.
We cover $\mathbb{R}^{n-1}$ by a finitely overlapping family of
slabs of dimensions $R^{2\delta} \times R^{1/2+\delta} \times \cdots \times R^{1/2+\delta}$
that are parallel to $R^\delta\tilde{\theta}^*$, and we denote this family by
\[
\mathcal{T}_\theta^\flat := \{T^\flat\}.
\]
Let $\{\eta_{T^\flat}\}$ be a smooth partition of unity subordinate to this cover.
Then
\[
\widehat{f}
= \sum_{\theta \in \Theta_{R^{-1/2}}}
  \sum_{T^\flat \in \mathcal{T}_\theta^\flat}
  \eta_{T^\flat}\,\widehat{f}_\theta .
\]
For each $\theta \in \Theta_{R^{-1/2}}$ and $T^\flat \in \mathcal{T}_\theta^\flat$,
the function
$(\eta_{T^\flat}\widehat{f}_\theta)^\vee
= \check{\eta}_{T^\flat} * f_\theta$
is essentially supported in $2\tilde{\theta}$.
Consequently,
\begin{equation}
\label{wave packet 1}
f
= \sum_{\theta \in \Theta_{R^{-1/2}}}
  \sum_{T^\flat \in \mathcal{T}_\theta^\flat}
  (\check{\eta}_{T^\flat} * f_\theta)\,\psi_{\tilde{\theta}}
  + \mathrm{RapDec}(R)\|f\|_2 ,
\end{equation}
where $\psi_{\tilde{\theta}}$ is a smooth bump function with
$\supp \psi_{\tilde{\theta}} \subset 3\tilde{\theta}$ and
$\psi_{\tilde{\theta}}(\xi)=1$ for all $\xi \in 2\tilde{\theta}$.

To each $\theta \in \Theta_{R^{-1/2}}$ and $T^\flat \in \mathcal{T}_\theta^\flat$
we associate the set
\[
T
:= \bigl\{
\bigl(\gamma^\lambda(u,x_n;\omega_\theta,1),x_n\bigr)
: u \in R^\delta T^\flat,\ -R \le x_n \le R
\bigr\}
\]
and the function
\begin{equation} \label{def fT}
f_T := (\check{\eta}_{T^\flat} * f_\theta)\,\psi_{\tilde{\theta}}.
\end{equation}
We write $\mathcal{T}_\theta := \{T : T^\flat \in \mathcal{T}_\theta^\flat\}$ and $\mathcal{T} := \bigcup_{\theta \in \Theta_{R^{-1/2}}} \mathcal{T}_\theta$, and for $T\in\mathcal T_\theta$ we write $\theta(T):=\theta$. With this notation, \eqref{wave packet 1} becomes
\[
f
= \sum_{\theta \in \Theta_{R^{-1/2}}}
  \sum_{T \in \mathcal{T}_\theta} f_T
  + \mathrm{RapDec}(R)\|f\|_2 .
\]
This is the wave packet decomposition of $f$ at scale $R$, and the functions $f_T$ are called \emph{wave packets}. The function $T^\lambda f_T$ is concentrated on $T$, in the following sense.
\begin{lemma}[{\cite[Lemma 3.12]{GW}}]\label{lemsuppT}
    The function $(T^\lambda f_T)\Id_{B_R^n(0)}$ is essentially supported in $ T$; that is, \[(T^\lambda f_T)\Id_{B_R^n(0)\setminus  T}=\rap(R)\|f\|_2.\]
\end{lemma}
By Plancherel's theorem, the wave packets are almost orthogonal.
\begin{lemma}\label{l2-orthogonalty}
    Let $f \in \mathcal{S}(\mathbb{R}^{n-1})$ with $\supp f \subset \mathbb{A}$, and let $\mathcal{T}' \subset \mathcal{T}$. Then
    \[
\|\sum_{T \in \mathcal{T}'}f_T\|_{L^2(\mathbb{R}^{n-1})}^2 \sim \sum_{T \in \mathcal{T}'} \|f_T\|_2^2.
\]
\end{lemma}

For each $T \in \mathcal{T}$ and each $-R \le t \le R$, the slice $T \cap \{x_n = t\}$ is essentially a box of dimensions $R^{2\delta} \times R^{1/2+\delta} \times \cdots \times R^{1/2+\delta}$. To make this precise, we use the following notion.

\begin{definition}[{\cite[Definition 3.5]{GW}}]
\label{defgaV}
Given $\xi_0\in \mathbb{A}$ and a box $V\subset B_R^{n-1}(0)$, define
\[
\Ga_V(\xi_0;R):=\{ (\ga^\la(u,x_n;\xi_0),x_n):u\in V,|x_n|\le R \}.
\]
The set $\Ga_V(\xi_0;R)$ is called a curved plank with base $V$ and direction $\xi_0$.
\end{definition}
In this terminology, each $T\in \mathcal{T}$ can be written as
    \[T= \Ga_{R^\de T^\flat}(\omega_\theta,1;R), \]
that is, $T$ is a curved plank with base $R^\de T^\flat$ and direction $(\omega_\theta,1)$. We call each such $T$ an \textbf{$R$-plank}.
\begin{lemma}[{\cite[Lemma 3.9]{GW}}]\label{geometric}
Let $s\in[R^{-1/2},1]$ and $\xi_0 \in \mathbb{A}$. Let $V\subset \mathbb{R}^{n-1}$ be a slab of dimensions $Rs^2\times Rs \times \dots\times Rs $. Then for $|t|\leq R$, each $t$-slice $\Ga_V(\xi_0;R)\cap \{x_n=t\} $ is comparable to a $Rs^2\times Rs\times \dots\times Rs$-slab.
\end{lemma}

\subsection{Elementary estimates} The following standard estimates will be used in the decoupling argument. Roughly speaking, they say that $|T^\lambda f_T|$ is essentially constant on $T$.
\begin{lemma} \label{essential constant}
    Let $T^{\lambda}$ be an oscillatory integral operator satisfying the cinematic curvature condition of type $\bfone$, and assume \eqref{assumpetion 1}. Let $q \geq 2$, and let $g \in \mathcal{S}(\mathbb{R}^{n-1})$ be supported in $\mathbb{A}$. Then the estimate
    \[ \| T^\lambda g\|_{L^q(\mathbb{R}^{n-1} \times \{t\})}^q \lesssim_{\phi,q} |\supp g|^{\frac{q}{2}-1}\|g\|_2^q\]
    holds uniformly for all $t \in \mathbb{R}$.
\end{lemma}
We shall use the following two consequences of this lemma.
\begin{corollary} \label{essential constant decoupling}
    Let $T^{\lambda}$ be an oscillatory integral operator satisfying the cinematic curvature condition of type $\bfone$, and assume \eqref{assumpetion 1}. Let $f \in \mathcal{S}(\mathbb{R}^{n-1})$ with $\supp f \subset \mathbb{A}$, let $T \in \mathcal{T}$, and let $q = 2 + \frac{4}{n-2}$. Then
    \[
\|T^{\lambda}f_T\|_{L^q(B_{2R}^n(0))} \lesssim_{\phi}\|f_T\|_{L^{2}(\mathbb{R}^{n-1})}.
\]
\end{corollary}
\begin{proof} Since $f_T$ is supported in $3\tilde{\theta}$ for some $\theta \in \Theta_{R^{-1/2}}$, we have $|\supp f_T| \lesssim R^{-\frac{n-2}{2}}$. Thus Lemma \ref{essential constant} gives
\[\|T^{\lambda}f_T\|_{L^q(B_{2R}^n(0))}^q \leq \int_{-2R}^{2R} \|T^\lambda f_T\|_{L^q(\mathbb{R}^{n-1} \times \{t\})}^q dt \lesssim_{\phi}\|f_T\|_{L^{2}(\mathbb{R}^{n-1})}^q.\qedhere\]
\end{proof}
\begin{corollary} \label{l2 estimate}
    Let $T^{\lambda}$ be an oscillatory integral operator satisfying the cinematic curvature condition of type $\bfone$, and assume \eqref{assumpetion 1}. Then the estimate
    \[
\|T^{\lambda}f\|_{L^2(Q)}^2 \lesssim_{\phi}R^{1/2}\|f\|_{L^{2}(\mathbb{R}^{n-1})}^2
\]
    holds uniformly for all $f \in \mathcal{S}(\mathbb{R}^{n-1})$ with $\supp f \subset \mathbb{A}$ and all cubes $Q$ of side length $R^{1/2}$.
\end{corollary}
\begin{proof}
Apply Lemma \ref{essential constant} with $q=2$ and integrate over the interval of length $R^{1/2}$ that is the projection of $Q$ to the $x_n$-axis.
\end{proof}

Lemma \ref{essential constant} is a consequence of the following result of H\"ormander, which generalizes the Hausdorff--Young inequality.

\begin{lemma}[{\cite[Theorem 1.1]{H}}] \label{lem:hormander}
    Let $q \geq 2$, and let $S^\lambda$ be the oscillatory integral operator
    \[ S^\lambda h(x):= \int e^{-2\pi \lambda i \bar{\phi}(x,y)} \bar{a}(x,y) h(y) dy,\]
    where $\bar{\phi}\colon B^{n-1}_1(0) \times \mathbb{A} \to \mathbb{R}$ and $\bar{a}\colon B^{n-1}_1(0) \times \mathbb{A} \to \mathbb{R}^{+}$ satisfy
    \begin{align}
    \det \partial^2_{x,y}\bar{\phi}(x,y) > 0 \qquad \text{for all $ (x,y)\in B^{n-1}_1(0)\times \mathbb{A}$,} \label{sh1} \\
    \|\partial_{x,y}^\beta \bar{a}\|_{L^\infty(B^{n-1}_1(0) \times \mathbb{A})} \lesssim_\beta 1 \qquad \qquad \text{for all $1 \leq |\beta| \leq N-1$.}\label{sh2}
    \end{align}
    Then
    \[
\|S^\lambda g\|_{L^q(\mathbb{R}^{n-1})} \lesssim_{\phi,q} \lambda^{-\frac{n-1}{q}} \|g\|_{L^{q'}(\mathbb{R}^{n-1})},
\]
    where $q'$ is the H\"older conjugate of $q$.
\end{lemma}
\begin{proof} [Proof of Lemma \ref{essential constant}]
    Fix $t$ and define
    \[S^\lambda g(u) := T^\lambda g(\lambda u,t).\]
    Since $T^\lambda$ satisfies \textup{(H2)} and \eqref{assumpetion 1}, the conditions \eqref{sh1} and \eqref{sh2} hold. Consequently, Lemma \ref{lem:hormander} yields
    \[ \| T^\lambda g\|_{L^q(\mathbb{R}^{n-1} \times \{t\})}^q = \lambda^{n-1} \|S^\lambda g\|_{L^q(\mathbb{R}^{n-1})}^q \lesssim_{\phi,q} \|g\|_{L^{q'}(\mathbb{R}^{n-1})}^q.\]
    The lemma now follows from H\"older's inequality.
\end{proof}

\subsection{A mixed-norm estimate}
The $L^2$ norm is well adapted to wave packet decompositions. We therefore replace the $L^r$ norm of $f$ by a quantity that involves only $L^2$ norms. For a cap $\theta$, let
\[
\|f\|_{L^{2}_\mathrm{avg}(3\tilde{\theta})}:= \left(\frac{1}{|3\tilde{\theta}|}\int_{3\tilde{\theta}} |f|^2\right)^{1/2}.
\]
\begin{proposition}\label{reduction 2}
Let $T^{\lambda}$ be an oscillatory integral operator satisfying the cinematic curvature condition of type $\bfone$, and assume \eqref{assumpetion 1} and \eqref{assumption2}. Let $p=p_n$ and $r=r_n$ be the exponents defined in \eqref{eq:new-endpoint-pair}. Then, for every $\varepsilon > 0$, the estimate
\[
\|T^{\lambda}f\|_{L^p(B_R^n(0))} \lesssim_{\varepsilon,p,\phi}R^\varepsilon \|f\|_{L^{2}(\mathbb{R}^{n-1})}^\frac{2}{r} \max_{\theta \in \Theta_{R^{-1/2}}}\|f\|_{L^{2}_\mathrm{avg}(3\tilde{\theta})}^{1-\frac{2}{r}}
\]
holds uniformly for $f \in \mathcal{S}(\mathbb{R}^{n-1})$ with $\supp f \subset \mathbb{A}$ and $1 \leq R \leq \lambda^{1-\varepsilon}$.
\end{proposition}

Proposition~\ref{reduction 2} implies Proposition~\ref{main theorem type 1}: the restricted-type interpolation argument used in \cite{G,Wang} gives the $L^{r_n}\to L^{p_n}$ estimate, and interpolation with the trivial $L^1\to L^\infty$ estimate, together with H\"older's inequality, then gives the full range of exponents. Working at the endpoint $(p_n,r_n)$ has the advantage that $2<r\leq p$ throughout the induction.

We record two properties of the maximum of the $L^2$ averages over the $R^{-1/2}$-caps. The first is a lower bound in terms of the number of wave packets.
\begin{lemma} \label{norm_lower_bound}
    Let $f \in \mathcal{S}(\mathbb{R}^{n-1})$ be supported in $\mathbb{A}$, and let $\mathbb{T} \subset \mathcal{T}$ be a nonempty family such that
    \[
\|f_T\|_2 \sim u \qquad \text{for all $T \in \mathbb{T}$.}
\]
    Then, for every $\gamma \in [0,1]$,
    \[\max_{\theta \in \Theta_{R^{-1/2}}}\|f\|_{L_\mathrm{avg}^2(3\tilde{\theta})}^2 +\rap(R)\|f\|_2^2 \gtrsim u^2 R^{\frac{n-2}2(1-\gamma)}(\#\mathbb{T})^\gamma.\]
\end{lemma}
\begin{proof}
Choose $\theta' \in \Theta_{R^{-1/2}}$ that maximizes $\#(\mathbb{T} \cap \mathcal{T}_\theta)$. Since $\mathbb{T}$ is nonempty and there are $\sim R^{\frac{n-2}{2}}$ caps, we have
\begin{equation} \label{l3.5lb}
    \#(\mathbb{T} \cap \mathcal{T}_{\theta'}) \gtrsim \max \{1,(\#\mathbb{T})R^{-\frac{n-2}{2}}\} \geq \ (\#\mathbb{T})^\gamma R^{-\gamma\frac{n-2}2}.
\end{equation}
The function $f_{\theta'}$ is supported in $3\tilde{\theta}'$, and $|3\tilde\theta'|\sim R^{-\frac{n-2}{2}}$. Hence, by Plancherel's theorem,
\[
\begin{aligned}
\max_{\theta \in \Theta_{R^{-1/2}}}\|f\|_{L_\mathrm{avg}^2(3\tilde{\theta})}^2 &\geq \max_{\theta \in \Theta_{R^{-1/2}}}\|f_{\theta}\|_{L_\mathrm{avg}^2(3\tilde{\theta})}^2 \gtrsim R^\frac{n-2}{2}\|f_{\theta'}\|_{L^2(\mathbb{R}^{n-1})}^2 \\ &\gtrsim u^2 R^\frac{n-2}{2}\, \#(\mathbb{T} \cap \mathcal{T}_{\theta'}) -\rap(R)\|f\|_2^2.
\end{aligned}
\]
Combining this with \eqref{l3.5lb} completes the proof.
\end{proof}

The second property is used in the induction on scales.
\begin{lemma} \label{norm_l2_local}
    Let $f \in \mathcal{S}(\mathbb{R}^{n-1})$ be supported in $\mathbb{A}$, let $\mathbb{T} \subset \mathcal{T}$, and let $g:=\sum_{T \in \mathbb{T}}f_T$. Then
    \[\max_{\theta \in \Theta_{R^{-1/2}}}\|g\|_{L_\mathrm{avg}^2(3\tilde{\theta})}^2 \lesssim \max_{\theta \in \Theta_{R^{-1/2}}}\|f\|_{L_\mathrm{avg}^2(3\tilde{\theta})}^2 + \rap(R)\|f\|_2^2.\]
\end{lemma}
\begin{proof}
    Fix $\theta \in \Theta_{R^{-1/2}}$. Only the wave packets $f_T$ with $3\theta(T) \cap 3\theta \neq \emptyset$ contribute to $g$ on $3\tilde\theta$, so Lemma \ref{l2-orthogonalty} gives
\begin{multline}\label{hahaha}
    \|g\|_{L^2(3\tilde{\theta})}^2 \leq \|\sum_{T \in \mathbb{T}, 3\theta(T) \cap 3\theta \neq \emptyset}f_T\|_{L^2(\mathbb{R}^{n-1})}^2 \lesssim \\ \|\sum_{3\theta' \cap 3\theta \neq \emptyset}f_{\theta'}\|_{L^2(\mathbb{R}^{n-1})}^2 +\rap(R)\|f\|_2^2.
\end{multline}
Since the caps in $\Theta_{R^{-1/2}}$ have bounded overlap and each $f_{\theta'}$ is supported in $3\tilde{\theta}'$, we have
\[\|\sum_{3\theta' \cap 3\theta \neq \emptyset}f_{\theta'}\|_{L^2(\mathbb{R}^{n-1})}^2 \lesssim \|f\|_{L^2(9\tilde{\theta})}^2.\]
The box $9\tilde\theta$ is covered by $O(1)$ boxes $3\tilde\theta''$ with $\theta'' \in \Theta_{R^{-1/2}}$. By pigeonholing, there is such a cap $\theta''$ with
\[\|\sum_{3\theta' \cap 3\theta \neq \emptyset}f_{\theta'}\|_{L^2(\mathbb{R}^{n-1})}^2 \lesssim \|f\|_{L^2(3\tilde{\theta}'')}^2.\]
Combining this with \eqref{hahaha} and dividing by $|3\tilde\theta|$, we obtain the lemma.
\end{proof}

\subsection{Lorentz rescaling}
\label{Lorentz-rescaling-section}
Let $T^{\lambda}$ be an oscillatory integral operator satisfying the cinematic
curvature condition of type $\mathbf{1}$, and assume \eqref{assumpetion 1} and \eqref{assumption2}. Let $g \in \mathcal{S}(\mathbb{R}^{n-1})$ be a function with
$\supp g \subset \A\cap\tilde\tau$ for some $\tau \in \Theta_{K^{-1}}$; that is, the support of $g$ lies over a single $K^{-1}$-cap. To estimate
$\|T^\lambda g\|_{L^p(B_R^n)}$ by induction on scales, we use a rescaling
that transforms $T^\lambda g$ into
$\widetilde{T}^{\lambda/K^2}\widetilde{g}$, where
$\widetilde{T}^{\lambda/K^2}$ is again an oscillatory integral operator satisfying the cinematic curvature condition and $\widetilde{g}$ is a function supported in $\mathbb{A}$. Under this rescaling, the scale of the domain of integration is reduced from $R$ to $R/K^2$.

The precise statement is the following lemma, whose proof follows \cite[Section 2.5]{BHS} and \cite[Section 5.1]{GW}.

\begin{lemma}\label{rescaling}
Let $n \geq 3$ and $2 < r \leq p$. Suppose that there is a constant $C>0$ such that, for every function $\widetilde g$ with $\supp \widetilde g \subset \mathbb{A}$ and every oscillatory integral operator $T^{\lambda/K^2}$ satisfying the cinematic curvature condition of type $\mathbf{1}$, the estimate
\[
\|T^{\lambda/K^2} \widetilde g\|_{L^p(B_{R/K^2}^n)}
\le
C\,\|\widetilde g\|_2^\frac2r
\max_{\theta \in \Theta_{K R^{-1/2}}}
\|\widetilde g\|_{L^2_{\mathrm{avg}}(3\tilde{\theta})}^{1-\frac2r}
+ \rap(R/K^2)\,\|\widetilde g\|_2
\]
holds whenever $a$ satisfies \eqref{assumpetion 1}.

Then, for every function $g \in \mathcal{S}(\mathbb{R}^{n-1})$ with $\supp g \subset \A\cap\tilde\tau$ for some $\tau \in \Theta_{K^{-1}}$, and every oscillatory integral operator $T^\lambda$ satisfying the cinematic curvature condition of type $\mathbf{1}$, the estimate
\[
\|T^\lambda g\|_{L^p(B_R^n(0))}
\lesssim_{\phi}
C\,K^{\frac{n}{p}-(n-2)(1-\frac1r)}\,
\|g\|_2^\frac2r
\max_{\theta \in \Theta_{R^{-1/2}}}
\|g\|_{L^2_{\mathrm{avg}}(3\tilde{\theta})}^{1 - \frac2r}
+ \rap(R)\,\|g\|_2
\]
holds whenever \eqref{assumpetion 1} and \eqref{assumption2} are satisfied.
\end{lemma}
\begin{proof}
Let $T^\lambda$ be an oscillatory integral operator satisfying the cinematic curvature condition of type $\mathbf{1}$, and assume \eqref{assumpetion 1} and \eqref{assumption2}.

Fix $\tau \in \Theta_{K^{-1}}$, and write $\xi_\tau:=(\om_\tau,1)$. Let $V_{\tau,R}\subset B_R^{n-1}$ be the isotropic
$RK^{-2}$-dilate of the dual box $\widetilde{\tau}^*$; thus $V_{\tau,R}$ is a box of dimensions
$RK^{-2}\times RK^{-1}\times \cdots \times RK^{-1}$. We cover $\mathbb{R}^{n-1}$
by finitely overlapping translates of $V_{\tau,R}$ and denote the resulting
collection by $\mathcal{B}_\tau^\flat=\{\Box^\flat\}$. For each
$\Box^\flat \in \mathcal{B}_\tau^\flat$, we define the curved plank
\[
\Box:=\Ga_{\Box^\flat}(\xi_\tau;R)=\{(\ga^\la(v,t;\xi_\tau),t):v\in \Box^\flat,|t|\le R \},
\]
and we let $\mathcal{B}_\tau := \{\Box : \Box^\flat \in \mathcal{B}_\tau^\flat\}$.
The sets in $\mathcal{B}_\tau$ cover $B_R^n(0)$ with finite overlap, and
each
\[
T \in \bigcup_{\theta \in \Theta_{R^{-1/2}},\, 2\theta \cap \tau \neq \emptyset}
\mathcal{T}_\theta
\]
is contained in $O(1)$ of the sets $\Box \in \mathcal{B}_\tau$. We assign each
such $T$ to one set $\Box$ that contains it, and we denote by $\mathcal{T}_\Box$
the collection of the planks assigned to $\Box$. This gives a
partition
\[
\bigcup_{\theta \in \Theta_{R^{-1/2}}, 2\theta\cap \tau \neq \emptyset}\mathcal{T}_\theta=\bigcup_{\Box}\mathcal{T}_\Box.
\]
Let $g \in \mathcal{S}(\mathbb{R}^{n-1})$ with $\supp g \subset \A\cap\tilde\tau$. Then
\[T^\la g=\sum_{\theta \in \Theta_{R^{-1/2}},2\theta\cap \tau \neq \emptyset}T^\la g_\theta=\sum_\Box T^\la g_\Box + \rap(R)\,\|g\|_2,\]
where
\[g_\Box := \sum_{T\in\mathcal{T}_\Box} g_T.\]
Since the sets in $\mathcal{B}_\tau$ have finite overlap, we have
\begin{equation}\label{rescplugin}
    \|T^\la g\|_{L^p(B_R^n(0))}^p\lesssim \sum_{\Box} \|T^\la g_\Box\|_{L^p(\Box)}^p + \rap(R)\,\|g\|_2^p.
\end{equation}
We now apply the Lorentz rescaling to each $g_\Box$. Define $\tilde{g_\Box}$ by
\begin{equation} \label{rescaling3}
    \tilde{g_\Box}(\eta):=K^{-(n-2)}g_\Box(\eta_{n-1}\om_\tau+K^{-1}\eta',\eta_{n-1}).
\end{equation}
As shown in the proof of \cite[Lemma 2.3]{BHS}, we have
\begin{equation}\label{rescaling1}
    T^\la g_\Box\circ \Upsilon_{\om_\tau}^\la\circ D_K= \wt T^{\la/K^2}\tilde{g_\Box},
\end{equation}
where
\[
\begin{aligned}
\Upsilon^\la_{\om_\tau}(u,x_n)
&:= (\la \Phi(u/\lambda,x_n/\lambda;\om_\tau,1),x_n),\\
D_K(x'',x_{n-1},x_n)&:=(Kx'',x_{n-1},K^2x_n),
\end{aligned}
\]
and
\[\wt T^{\la/K^2}\tilde{g_\Box}(x):=\int_{\R^{n-1}}e^{i\wt\phi^{\la/K^2}(x;\eta)}\wt a^\la(x;\eta)\tilde{g_\Box}(\eta)\mathrm{d}\eta.\]
Here the phase function $\wt \phi$ and the amplitude function $\wt a$ are given by
\[
\begin{gathered}
\wt \phi(x;\eta) := \langle x',\eta\rangle \\
+ \int_0^1 (1-r)\langle \partial^2_{\xi'\xi'}\phi(\Upsilon_{\om_\tau}(D'_{K^{-1}}x',x_n);\eta_{n-1}\om_\tau +rK^{-1}\eta',\eta_{n-1})\eta',\eta'\rangle\mathrm{d} r
\end{gathered}
\]
and
\[\wt a(x;\eta):= a(\Upsilon_{\om_\tau}(D'_{K^{-1}}x';x_n);\eta_{n-1}\om_\tau+K^{-1}\eta',\eta_{n-1}),\]
where
\[D'_{K^{-1}}(x'',x_{n-1}):=(K^{-1}x'',K^{-2}x_{n-1}),\]
\[\Upsilon_{\om_\tau}(u,x_n):= (\Phi(u,x_n;\om_\tau,1),x_n).\]
For each $\Box \in \mathcal{B}_\tau$, the set $(\Upsilon^\lambda_{\omega_\tau} \circ D_K)^{-1}(\Box)$ is essentially a ball of radius $R/K^2$, which we denote by $B^n_{R/K^2,\Box}$. By \eqref{rescaling1}, we can rewrite \eqref{rescplugin} as
\begin{equation}\label{rescaling2}
    \|T^\la g\|_{L^p(B_R^n(0))}^p\lesssim K^n \sum_{\Box} \|\wt T^{\la/K^2}\tilde{g_\Box}\|_{L^p(B^n_{R/K^2,\Box})}^p + \rap(R)\,\|g\|_2^p.
\end{equation}

By \cite[Section 2.5]{BHS}, the rescaled operator $\wt T^{\la/K^2}$ is of type $(1,1,C_\phi)$, where the constant $C_\phi$ depends only on $\phi$, and the rescaled amplitude $\wt{a}$ satisfies \eqref{assumpetion 1} up to a constant depending only on $\phi$. After a further rescaling and normalization, we may therefore apply the hypothesis of the lemma to each of the functions $\tilde{g_\Box}$:
\begin{multline}\label{rescaling4}
     \|\wt T^{\la/K^2}\tilde{g_\Box}\|_{L^p(B^n_{R/K^2,\Box})}^p \lesssim_\phi C^p\,\|\tilde{g_\Box}\|_2^\frac{2p}{r}
\max_{\theta \in \Theta_{K R^{-1/2}}}
\|\tilde{g_\Box}\|_{L^2_{\mathrm{avg}}(3\tilde{\theta})}^{\,p-\frac{2p}{r}} \\
+ \rap(R/K^2)\,\|\tilde{g_\Box}\|_2^p.
\end{multline}
By \eqref{rescaling3} and Lemma \ref{norm_l2_local},
\begin{equation} \label{rescaling5}
    \|\tilde{g_\Box}\|_2^2 = K^{-(n-2)} \|g_\Box\|_2^2
\end{equation}
and
\begin{equation} \label{rescaling6}
\max_{\theta \in \Theta_{K R^{-1/2}}} \|\tilde{g_\Box}\|_{L^2_{\mathrm{avg}}(3\tilde{\theta})}^2 \lesssim K^{-2(n-2)}\max_{\theta \in \Theta_{R^{-1/2}}}
\|g\|_{L^2_{\mathrm{avg}}(3\tilde{\theta})}^2+\rap(R)\|g\|_2^2.
\end{equation}
Combining \eqref{rescaling2}, \eqref{rescaling4}, \eqref{rescaling5}, and \eqref{rescaling6}, we see that
\[
\begin{gathered}
\|T^\la g\|_{L^p(B_R^n(0))}^p
\lesssim_\phi C^pK^{n-(n-2)(p-\frac pr)}
\max_{\theta \in \Theta_{R^{-1/2}}}
\|g\|_{L^2_{\mathrm{avg}}(3\tilde{\theta})}^{p-\frac{2p}{r}}
\sum_{\Box} \|g_\Box\|_2^\frac{2p}{r}\\
+\rap(R)\,\|g\|_2^p.
\end{gathered}
\]
Since $r \leq p$, Lemma \ref{l2-orthogonalty} gives $\sum_\Box\|g_\Box\|_2^{2p/r}\lesssim\|g\|_2^{2p/r}$, and the lemma follows.
\end{proof}

\subsection{Simplification of the phase function} By the argument in \cite[Proof of Lemma 5.4]{GW}, Proposition \ref{reduction 2} can be reduced to the case in which
\begin{equation} \label{assumption 3}
    \phi(x;\xi) = \langle x',\xi\rangle + x_n\frac{h(\xi')}{\xi_{n-1}} + \cE(x;\xi'/\xi_{n-1})\xi_{n-1},
\end{equation}
where $h$ is a quadratic form and the error term satisfies
\begin{equation} \label{assumption 4}
\cE(x;\om) \lesssim |x_n| |\om|^3+|x|^2|\om|^2.
\end{equation}
Thus it suffices to prove the following.
\begin{proposition}\label{redution 3}
Let $T^{\lambda}$ be an oscillatory integral operator satisfying the cinematic curvature condition of type $\bfone$, and assume \eqref{assumpetion 1}, \eqref{assumption2}, \eqref{assumption 3}, and \eqref{assumption 4}. Let $p=p_n$ and $r=r_n$ be the exponents defined in \eqref{eq:new-endpoint-pair}. Then, for every $\varepsilon > 0$, the estimate
\[
\|T^{\lambda}f\|_{L^p(B_R^n(0))} \lesssim_{\varepsilon,p,\phi}R^\varepsilon \|f\|_{L^{2}(\mathbb{R}^{n-1})}^\frac2r \max_{\theta \in \Theta_{R^{-1/2}}}\|f\|_{L^{2}_\mathrm{avg}(3\tilde{\theta})}^{1-\frac2r}
\]
holds uniformly for $f \in \mathcal{S}(\mathbb{R}^{n-1})$ with $\supp f \subset \mathbb{A}$ and $1 \leq R \leq \lambda^{1-\varepsilon}$.
\end{proposition}

\subsection{Broad norms}\label{sectionbroadnorm} We recall the definition of the broad norm, which plays a central role in the decoupling argument. Related definitions can be found in \cite{BG,GW,G,GHI,OW}.

Let $f \in \mathcal{S}(\mathbb{R}^{n-1})$ with $\supp f \subset \mathbb{A}$, and let $T^{\lambda}$ be an oscillatory integral operator satisfying the cinematic curvature condition of type $\mathbf{1}$. Assume \eqref{assumpetion 1}, \eqref{assumption2}, \eqref{assumption 3}, and \eqref{assumption 4}, and let $p \geq 2$.

Fix an integer $A\ge 1$, and cover $B_\la^n(0)$ by a finitely overlapping collection $\mathcal{B}$ of balls of radius $K^2$. For each ball $B_{K^2}^n(y) \in \mathcal{B}$ with $y \in B^n_\la(0)$, define
\[
\mu(B_{K^2}^n(y)):=\min_{v_1,\dots,v_A \in S^{n-1}}\max_{\substack{\tau\in \Theta_{K^{-1}}\\ \ang(G_0(y/\lambda;\om_\tau,1),v_i)\gtrsim K^{-1}\text{for all $1\leq i \leq A$}}}\int_{B_{K^2}^n(y)}|T^\la f_\tau|^p.
\]
Thus the maximum is taken over the $K^{-1}$-caps $\tau$ whose Gauss directions at $y$ are separated from the $A$ directions $v_1,\dots,v_A$. For $U\subset B_R^n(0)$, we define the broad norm over $U$ by
\[\|T^\la f\|_{\BL^p_A(U)}^p:=\sum_{B_{K^2}^n \in \mathcal{B}}\frac{|B_{K^2}^n \cap U|}{|B_{K^2}^n|}\mu(B_{K^2}^n).\]

The broad norm satisfies the following H\"older inequality; see
\cite[Section~4]{G2}.

\begin{lemma}\label{lemholder}
Let $T^{\lambda}$ be an oscillatory integral operator satisfying the cinematic curvature condition, and let $f\colon\mathbb{R}^{n-1} \to \mathbb{C}$. Suppose that $1\le p,p_1,p_2<\infty$ and $0\le \al_1,\al_2\le 1$ satisfy $\al_1+\al_2=1$ and
    \[\frac1p=\frac{\al_1}{p_1}+\frac{\al_2}{p_2}, \]
and that $A_1,A_2\ge 1$ and $A=A_1+A_2$.
Then
\[\|T^\la f\|_{\BL^p_A(U)}\lesssim \|T^\la f\|_{\BL^{p_1}_{A_1}(U)}^{\al_1}\|T^\la f\|_{\BL^{p_2}_{A_2}(U)}^{\al_2}.\]
\end{lemma}

\subsection{Refined decoupling} The refined decoupling inequality was introduced in \cite{GIOW} as a strengthening of the decoupling inequality. We shall use the following version, which is proved in \cite[Appendix A]{GW}.
\begin{theorem}[{\cite[Theorem A.3]{GW}}]\label{refdecthm}
Let $T^{\lambda}$ be an oscillatory integral operator satisfying the cinematic
curvature condition of type $\mathbf{1}$, and let
$1 \le R \le \lambda^{1-\varepsilon}$. Assume that
\eqref{assumpetion 1}, \eqref{assumption2}, \eqref{assumption 3}, and
\eqref{assumption 4} hold. Let $\mathbb{T} \subset \mathcal{T}$ and define
\[g := \sum_{T \in \mathbb{T}} f_T.\]
Let $w_{B_R}$ be a bump function with $w_{B_R}(x)=1$ for
$x \in B_R^n(0)$ and $\supp w_{B_R} \subset B_{2R}^n(0)$, and suppose that the norms $\|T^\lambda f_T\|_{L^q(w_{B_R})}$, $T \in \mathbb{T}$, are comparable to each other. Let $Y$ be a union of finitely overlapping $K^2$-balls
that meet $B_R^n(0)$, each of which meets at most $M$ planks of
$\mathbb{T}$. Then, for all $2 \le q \le \frac{2n}{n-2}$ and all
$\varepsilon > 0$, we have
\[
\|T^\lambda g\|_{L^q(Y)}^q
\lesssim_{\varepsilon,\phi}
R^\varepsilon
M^{\frac q2-1}
\sum_{T \in \mathbb{T}}
\|T^\lambda f_T\|_{L^q(w_{B_R})}^q
+ \rap(R)\,\|f\|_2^q.
\]
\end{theorem}

\section{Two-ends reduction}\label{section-algorithm}

Throughout this section and the next, $p=p_n$ and $r=r_n$ are the exponents defined in \eqref{eq:new-endpoint-pair}, and the operator $T^\lambda$ and the function $f$ satisfy the assumptions of
Proposition~\ref{redution 3}. The function $f$ and its wave packet decomposition at scale $R$ are fixed: each of the refinements below selects a subfamily of the wave packets of $f$, and no new wave packet decomposition is performed.

We use the two-ends algorithm of \cite{GW} in the simplified form of \cite[Section~3]{WangDensity}. Instead of the regularity parameters of \cite{GW}, we record the number of wave packets before and after each pruning step; together with the transverse covering estimate, this is all that is needed.

\subsection{Normalization and incidence counts}

We first normalize the wave packets. Only $R^{O(1)}$ wave packets have planks that meet $B^n_R(0)$, and the other wave packets contribute a rapidly decaying error. By dyadic pigeonholing of the numbers $\|f_T\|_2$ and the triangle inequality, at the cost of a factor $(\log R)^{O(1)}$ we may restrict attention to a family $\mathbb T^{(0)}\subset\mathcal T$ and a number $u>0$ such that
\[
\|f_T\|_2\sim u\qquad(T\in\mathbb T^{(0)});
\]
the wave packets with very small $L^2$ norm again contribute a rapidly decaying error. Since this pigeonholing is made before the two-ends algorithm, every subfamily that is selected later has the same normalization. The norms on the right-hand side of the estimate of Proposition~\ref{redution 3} remain those of the original function $f$.

We fix compatible grids of cubes at the scales $K^2$, $R^{1/2}$, and $R/\Kc$,
changing these scales by bounded factors if necessary. When we count incidences, a $K^2$-ball is a half-open cube of the finest grid; in the analytic estimates we use balls of comparable size. Thus every $K^2$-ball
$B$ lies in a unique $R^{1/2}$-cube $Q$, and every $Q$ lies in
a unique horizontal region
\[
 \omega=\mathbb R^{n-1}\times I,\qquad |I|\sim R/\Kc.
\]
We write $\Omega$ for the partition into horizontal regions. From the outset, all planks are enlarged by the fixed factors that are needed to neglect the tails of the wave packets, and the same enlarged planks are used in all incidence counts and geometric estimates.

A \emph{shaded incidence triple} $(\mathbf B,\mathbf T;\mathcal G)$ consists
of a family $\mathbf B$ of $K^2$-balls, a family $\mathbf T$ of
$R$-planks, and a map $\mathcal G:\mathbf T\to P(\Omega)$, where $P(\Omega)$ is the set of all subsets of $\Omega$. The union $\bigcup\mathcal G(T)$ is the shading of the plank $T$: only the part of $T$ inside it is taken into account. Put
\[
 \mathbf T(B;\mathcal G)
 =\{T\in\mathbf T:T\cap B\ne\varnothing,
       \ B\subset\textstyle\bigcup\mathcal G(T)\},
\]
\[
 \mathbf B(T;\mathcal G)
 =\{B\in\mathbf B:T\cap B\ne\varnothing,
       \ B\subset\textstyle\bigcup\mathcal G(T)\},
\]
and define
\[
 \mathcal I(\mathbf B,\mathbf T;\mathcal G)
 =\sum_{B\in\mathbf B}\#\mathbf T(B;\mathcal G),
 \qquad
 L^p(\mathbf B,\mathbf T;\mathcal G;f)
 =\sum_{B\in\mathbf B}
   \left\|\sum_{T\in\mathbf T(B;\mathcal G)}T^\lambda f_T
   \right\|_{L^p(B)}^p.
\]

\begin{definition}[Broad cardinality]\label{defbbr}
For a family $\mathbf T$ of $R$-planks and an integer $a\geq1$, let
\[
 \#_{\mathrm{br},a}\mathbf T
 :=\min_{\tau_1,\ldots,\tau_a\in\Theta_{K^{-1}}}
 \#\{T\in\mathbf T:
       \theta(T)\not\subset\textstyle\bigcup_{j=1}^a4\tau_j\}.
\]
We write $\#_{\mathrm{br}}=\#_{\mathrm{br},1}$.
\end{definition}
Thus $\#_{\mathrm{br},a}\mathbf T$ is the least number of planks of $\mathbf T$ that remain after the planks with directions in $a$ enlarged $K^{-1}$-caps have been removed.

Let
\[
 A_0:=\left\lceil C_0(1+\log K)^2\right\rceil,
 \qquad A_*:=C_*A_0.
\]
Here $C_0$ is a large constant, depending only on the dimension
and on the bounds in the normalization of the phase, such that $A_0$ satisfies the condition on the broadness parameter in \cite{WangDensity}; and $C_*$ is a fixed constant so large that every $K^{-1}$-cap whose cutoff function meets a union of
$A_0$ enlarged $K^{-1}$-caps can be excluded from a broad norm with parameter $A_*$. Factors that are polynomial in $A_0$ are bounded by powers of $\log R$, and they are absorbed into the losses below.

We shall apply $\#_{\mathrm{br},A_0}$ to the wave packets through a $K^2$-ball, and $\#_{\mathrm{br}}$ to the wave packets through an $R^{1/2}$-cube.
As in \cite{WangDensity}, the broad cardinality at a ball has two roles: it is an upper bound for the multiplicity in the refined decoupling estimate, and a lower bound for the multiplicity in the transverse covering estimate.

It may happen that the broad cardinality of the wave packets through a ball $B$ is zero, that is, that their directions lie in $A_0$ enlarged
$K^{-1}$-caps. Then their contribution to the $\BL^p_{A_*}$ norm on $B$ vanishes, up to
negligible tails, and we are in a narrow situation: after localizing to the horizontal regions, we apply the induction on scales based on Lorentz rescaling, as in Section~\ref{subsec:two-ends-narrow}. The localization produces
at most $\Kc^{O(1)}$ terms, and this loss is absorbed by the gain of the induction, because $\Kc\ll K$. The same applies when, in a cube $Q$, all the wave packets under consideration have their directions in a single enlarged $K^{-1}$-cap. We call such a case a \emph{terminal narrow class}. Whenever a terminal narrow class occurs in the algorithm below, it is treated in this way; in the description of the algorithm we concentrate on the remaining classes.

\subsection{Algorithm 1: a single refinement step}

Let $(\mathbf B,\mathbf T;\mathcal G)$ be a shaded incidence triple for which the local $L^p$ norms
$\bigl\|\sum_{T\in\mathbf T(B;\mathcal G)}T^\lambda f_T\bigr\|_{L^p(B)}$, $B\in\mathbf B$, are comparable to each other. The refinement step has three parts.

\emph{Step 1: the multiplicity $\mu$.} We sort the balls $B\in\mathbf B$ according to the broad cardinality
$\#_{\mathrm{br},A_0}\mathbf T(B;\mathcal G)$, which is either zero or lies in a dyadic interval $[\mu,2\mu)$ with $\mu\geq1$, and we select a class that carries the largest part of $L^p(\mathbf B,\mathbf T;\mathcal G;f)$. If this is the class with broad cardinality zero, we have a terminal narrow class. Otherwise we obtain
$\mu\geq1$ and $\mathbf B_1\subset\mathbf B$ such that
\[
\mu\leq\#_{\mathrm{br},A_0}\mathbf T(B;\mathcal G)<2\mu
 \qquad(B\in\mathbf B_1),
\]
and
\[
L^p(\mathbf B,\mathbf T;\mathcal G;f)
 \lesssim(\log R)^{O(1)}
 L^p(\mathbf B_1,\mathbf T;\mathcal G;f).
\]

\emph{Step 2: pruning of the wave packets.} For dyadic numbers $W$ and $\beta$, set
\[
\mathcal G_W(T)
 :=\{\omega\in\mathcal G(T):
       \#\{B\in\mathbf B_1(T;\mathcal G):B\subset\omega\}
       \sim W\},
\]
\[
\mathbf T_{W,\beta}
 :=\{T\in\mathbf T:\#\mathcal G_W(T)\sim\beta\}.
\]
Thus $\mathcal G_W(T)$ consists of the horizontal regions of the shading in which $T$ meets about $W$ balls of $\mathbf B_1$, and $\mathbf T_{W,\beta}$ consists of the planks that have about $\beta$ such regions. By the triangle inequality and dyadic pigeonholing, there are $W$, $\beta$,
and a subcollection $\mathbf B_2\subset\mathbf B_1$ such that the norms $\bigl\|\sum_{T\in\mathbf T_{W,\beta}(B;\mathcal G_W)}T^\lambda f_T\bigr\|_{L^p(B)}$, $B\in\mathbf B_2$, are comparable to each other and
\[
L^p(\mathbf B_1,\mathbf T;\mathcal G;f)
 \lesssim(\log R)^{O(1)}
 L^p(\mathbf B_2,\mathbf T_{W,\beta};\mathcal G_W;f).
\]

\emph{Step 3: the cubes $Q$.} For each cube $Q$ of the $R^{1/2}$-grid, consider the wave packets that pass through the balls of $\mathbf B_2$ in $Q$, before and after the pruning:
\[
 \mathcal U_Q
 :=\bigcup_{\substack{B\in\mathbf B_2\\B\subset Q}}
       \mathbf T(B;\mathcal G),
 \qquad
 \mathcal V_Q
 :=\bigcup_{\substack{B\in\mathbf B_2\\B\subset Q}}
       \mathbf T_{W,\beta}(B;\mathcal G_W).
\]
We pigeonhole the cubes according to the dyadic values of $\#_{\mathrm{br}}\mathcal U_Q$ and $\#_{\mathrm{br}}\mathcal V_Q$. This
gives a collection $\mathcal Q$ of cubes and dyadic numbers $\bar m$ and $m$ such
that
\[
 \#_{\mathrm{br}}\mathcal U_Q\sim\bar m,
 \qquad
 \#_{\mathrm{br}}\mathcal V_Q\sim m
 \qquad(Q\in\mathcal Q).
\]
If $m=0$, this is another terminal narrow class. Otherwise $m\geq1$.
We keep every ball of $\mathbf B_2$ that lies in a cube of $\mathcal Q$, and set
\[
 \mathbf B_3
 :=\{B\in\mathbf B_2:B\subset Q\text{ for some }Q\in\mathcal Q\}.
\]
Then
\[
L^p(\mathbf B_2,\mathbf T_{W,\beta};\mathcal G_W;f)
 \lesssim(\log R)^{O(1)}
 L^p(\mathbf B_3,\mathbf T_{W,\beta};\mathcal G_W;f).
\]
Finally, put
\[
 l:=\sup_{T\in\mathbf T_{W,\beta}}
       \#\{Q\in\mathcal Q:T\cap Q\ne\varnothing\}.
\]

We record the properties of this construction that will be used. First, $\mu\lesssim\bar m$ and $m\lesssim\bar m$. Indeed, if $B\in\mathbf B_2$ and $B\subset Q$, then $\mathbf T(B;\mathcal G)\subset\mathcal U_Q$, and moreover $\mathcal V_Q\subset\mathcal U_Q$; the two inequalities follow, because the broad cardinality is monotone under inclusion and $\#_{\mathrm{br},A_0}\leq\#_{\mathrm{br}}$. Next, we have the incidence bounds
\begin{equation}\label{item3}
 \mathcal I(\mathbf B_1,\mathbf T;\mathcal G)
 \geq\#\mathbf B_1\,\mu
\end{equation}
and
\begin{equation}\label{item4}
 \mathcal I(\mathbf B_1,\mathbf T_{W,\beta};\mathcal G_W)
 \lesssim\#\mathbf T_{W,\beta}\,\beta W.
\end{equation}
The first holds because $\#\mathbf T(B;\mathcal G)\geq\mu$ for every $B\in\mathbf B_1$, and the second because every $T\in\mathbf T_{W,\beta}$ has about $\beta$ regions in $\mathcal G_W(T)$, each of which contains about $W$ balls of $\mathbf B_1(T;\mathcal G)$. Combining the three selections, we get
\begin{equation}\label{item5}
 L^p(\mathbf B,\mathbf T;\mathcal G;f)
 \lesssim(\log R)^{O(1)}
 L^p(\mathbf B_3,\mathbf T_{W,\beta};\mathcal G_W;f).
\end{equation}
Finally, the hairbrush argument of \cite[Section~6]{GW}, applied to the pruned family $\mathbf T_{W,\beta}$, gives the following estimate.
\begin{lemma}\label{item6}
If $\beta>100$, then
\[
 \#\mathbf B_1\gtrsim(K\Kc)^{-O(1)}Wlm.
\]
\end{lemma}
This estimate involves only the number $m$, which is defined by the pruned family; in particular, it does not require that the cubes of $\mathcal Q$ contain comparable numbers of selected balls.

\subsection{Algorithm 2: iteration}

We start with the family $\mathbb T^{(0)}$ fixed above and with $\mathcal G^{(0)}(T)=\Omega$ for $T\in\mathbb T^{(0)}$.
By pigeonholing the local $L^p$ norms, we obtain a collection $\mathcal B^{(0)}$ of $K^2$-balls, on which these norms are comparable to each other, such that
\begin{equation}\label{stage2}
 \|T^\lambda f\|_{L^p(B_R^n(0))}^p
 \lesssim(\log R)^{O(1)}
 L^p(\mathcal B^{(0)},\mathbb T^{(0)};\mathcal G^{(0)};f)
 +\mathrm{RapDec}(R)\|f\|_2^p.
\end{equation}
We then apply Algorithm~1 repeatedly, with
\[
 \mathcal B^{(i+1)}=\mathcal B_3^{(i)},\qquad
 \mathbb T^{(i+1)}
   =\mathbb T_{W^{(i)},\beta^{(i)}}^{(i)},\qquad
 \mathcal G^{(i+1)}=\mathcal G_{W^{(i)}}^{(i)}.
\]
If a terminal narrow class occurs, it is treated as described above. Otherwise we take the numbers $\mu^{(i)}$, $m^{(i)}$, and $\bar m^{(i)}$ to be the lower endpoints of the corresponding dyadic intervals. Since the balls of $\mathcal B^{(i+1)}=\mathcal B_3^{(i)}$ belong to $\mathcal B_2^{(i)}$, the families of Step~3 satisfy $\mathcal U_Q^{(i+1)}\subset\mathcal V_Q^{(i)}\subset\mathcal U_Q^{(i)}$ for every cube $Q$, and therefore
\[
 \bar m^{(i+1)}\leq m^{(i)}\leq\bar m^{(i)}.
\]
For the same reason, the incidence counts
\[
 I_i:=\mathcal I(\mathcal B_1^{(i)},\mathbb T^{(i)};
                         \mathcal G^{(i)})
\]
are nonincreasing in $i$. Both sequences consist of positive numbers and are bounded above by
$R^{O(1)}$. Consequently, there is an index $J\lesssim_\varepsilon1$
at which neither sequence decreases by more than the factor
$\kappa=R^{\varepsilon^{500}}$:
\[
\bar m^{(J)}\leq\kappa\bar m^{(J+1)}
                \leq\kappa m^{(J)},\qquad I_J\leq\kappa I_{J+1}.
\]
Indeed, if there were no such index, then at every step at least one of the two sequences would
decrease by a factor larger than $\kappa$, and this can happen only
$O_\varepsilon(1)$ times.

We fix such an index $J$. For the triple before this step we write
\[
 \mathbb T^{\mathrm{prior}}:=\mathbb T^{(J)},\qquad
 \mathcal G^{\mathrm{prior}}:=\mathcal G^{(J)},\qquad
 \mathbf B_2^{\mathrm{prior}}:=\mathcal B_2^{(J)},
\]
and for the triple after it
\[
 \mathbb T:=\mathbb T^{(J+1)},\quad
 \mathcal B:=\mathcal B_3^{(J)},\quad
 \mathcal G:=\mathcal G^{(J+1)},\quad
 \mathcal Q:=\mathcal Q^{(J)}.
\]
We omit the superscript $J$ from $\mu,W,\beta,m,\bar m,l$.
Then, for every $Q\in\mathcal Q$,
\begin{equation}\label{eq:stable-prior-count}
 \#_{\mathrm{br}}
 \bigcup_{\substack{B\in\mathbf B_2^{\mathrm{prior}}\\B\subset Q}}
 \mathbb T^{\mathrm{prior}}(B;\mathcal G^{\mathrm{prior}})
 \sim\bar m\lesssim\kappa m,
\end{equation}
and every $B\in\mathcal B$ satisfies
\begin{equation}\label{eq:stable-ball-count}
 \mu\leq\#_{\mathrm{br},A_0}
       \mathbb T^{\mathrm{prior}}(B;\mathcal G^{\mathrm{prior}})<2\mu.
\end{equation}
Note that these two statements concern the family $\mathbb T^{\mathrm{prior}}$, which in general is larger than $\mathbb T$. In Section~\ref{concludetheproof}, the function to be estimated is a sum of wave packets of $\mathbb T$, whereas the planks used in the covering argument are taken from $\mathbb T^{\mathrm{prior}}$.

\begin{lemma}\label{forfinal}
Suppose that $\beta>100$. Then
\begin{equation}\label{eq:hairbrush-cancellation}
 lm\mu\lesssim K^{O(1)}(\#\mathbb T),
 \qquad
 \mu^2l\lesssim K^{O(1)}(\#\mathbb T).
\end{equation}
\end{lemma}
\begin{proof}
By the choice of $J$ and by \eqref{item3} and \eqref{item4},
\[
 \#\mathcal B_1^{(J)}\mu
 \leq I_J\leq\kappa I_{J+1}
 \leq\kappa\mathcal I(\mathcal B_1^{(J)},\mathbb T;\mathcal G)
 \lesssim\kappa (\#\mathbb T)\beta W.
\]
Combining this with Lemma~\ref{item6}, we obtain
\[
lm\mu\lesssim(K\Kc)^{O(1)}\kappa\beta (\#\mathbb T)
             \lesssim K^{O(1)}(\#\mathbb T),
\]
since $\beta\lesssim\Kc$ and since $\Kc$ and $\kappa$ are bounded by $K$. This is the first inequality. The second follows from the first, because $\mu\lesssim\bar m\lesssim\kappa m$.
\end{proof}

The normalization of $\mathbb T^{(0)}$ is inherited by $\mathbb T$. Thus, by Lemma~\ref{l2-orthogonalty},
\begin{equation}\label{eq:uniform-packet-energy}
 (\#\mathbb T)u^2\sim\sum_{T\in\mathbb T}\|f_T\|_2^2\lesssim\|f\|_2^2,
\end{equation}
and Lemma~\ref{norm_lower_bound}, applied to the family $\mathbb T$, gives
for $0\leq\gamma\leq1$
\begin{equation}\label{eq:uniform-cap-lower-bound}
 \max_{\theta\in\Theta_{R^{-1/2}}}
 \|f\|_{L^2_{\mathrm{avg}}(3\widetilde\theta)}^2
 +\mathrm{RapDec}(R)\|f\|_2^2
 \gtrsim u^2R^{\frac{n-2}{2}(1-\gamma)}(\#\mathbb T)^\gamma.
\end{equation}
Finally, \eqref{stage2} and the $O_\varepsilon(1)$ applications of
\eqref{item5} give
\begin{equation}\label{eq:after-algorithm}
\|T^\lambda f\|_{L^p(B_R^n(0))}^p
 \lesssim(\log R)^{O_\varepsilon(1)}
 L^p(\mathcal B,\mathbb T;\mathcal G;f)
 +\mathrm{RapDec}(R)\|f\|_2^p.
\end{equation}
It remains to estimate the right-hand side of \eqref{eq:after-algorithm} when no terminal narrow class has occurred, so that $\mu\geq1$ and $m\geq1$. In the next section we do this separately in the one-end case
$\beta\leq100$ and in the two-ends case $\beta>100$.

\section{Completion of the proof}\label{concludetheproof}
We now prove Proposition~\ref{redution 3}. Recall that
\begin{equation}\label{prsetting}
 p=p_n,\qquad r=r_n,\qquad
 \frac{r}{r-1}=\frac{n-2}{n}p,
\end{equation}
so that $2<r\leq p$. We write the estimate to be proved in the form
\begin{equation}\label{eq:induction-claim}
\|T^{\lambda}f\|_{L^p(B_R^n(0))}^p\leq C_\e R^{p\e}\,\|f\|_2^{\frac{2p}{r}}\max_{\theta \in \Theta_{R^{-1/2}}}\|f\|_{L^{2}_\mathrm{avg}(3\tilde{\theta})}^{p-\frac{2p}{r}}+\rap(R)\,\|f\|_2^p,
\end{equation}
where $C_\e$ is a large constant, and we argue by induction on the scale $R$: we assume that \eqref{eq:induction-claim} holds at all scales smaller than $R/2$, and we prove it at scale $R$ with the same constant. For bounded $R$ there is nothing to prove.

We use the family $\mathbb T$, the collection $\mathcal B$, and the shading $\cG$ obtained in Section~\ref{section-algorithm}. As before, all wave packets in $\mathbb T$ satisfy $\|f_T\|_2\sim u$, and the norms on
the right-hand side of the estimates are those of the original function $f$.

We subdivide the horizontal regions into cubes $B_k$ of
side length comparable to $R/\Kc$, compatible with the grids of Section~\ref{section-algorithm}. Let $\omega_k$ be the horizontal
region that contains $B_k$, and set
\[
 \mathbb T_k:=\{T\in\mathbb T:T\cap B_k\ne\varnothing,
                       \ \omega_k\in\cG(T)\},
 \qquad f_k:=\sum_{T\in\mathbb T_k}f_T.
\]
If $B\in\mathcal B$ is a $K^2$-ball contained in $B_k$, then the planks of $\mathbb T_k$ that meet $B$ are exactly the planks of
$\mathbb T(B;\cG)$. Hence, up to a rapidly decaying error, $T^\lambda f_k$ agrees on $B$ with $\sum_{T\in\mathbb T(B;\cG)}T^\lambda f_T$, and \eqref{eq:after-algorithm} gives
\begin{equation}\label{cp1}
 \|T^\lambda f\|_{L^p(B_R^n)}^p
 \lesssim (\log R)^{O_\varepsilon(1)}
 \sum_k\|T^\lambda f_k\|_{L^p(B_k)}^p
 +\rap(R)\|f\|_2^p.
\end{equation}
We distinguish the one-end case $\beta\leq100$ and the two-ends case
$\beta>100$. The terminal narrow classes of Section~\ref{section-algorithm} are treated by the rescaling argument of Section~\ref{subsec:two-ends-narrow}.

\subsection{The one-end case} Suppose that $\beta\le 100$. Applying the induction hypothesis on each $B_k$, we obtain
\begin{multline}\label{one-end-case-1}
    \|T^\la f_k\|_{L^p(B_k)}^p\le C_\e \Bigl(\frac{R}{\Kc}\Bigr)^{p\e}\|f_k\|_2^\frac{2p}{r}\,\max_{\theta \in \Theta_{\Kc^{1/2} R^{-1/2}}}\|f_k\|_{L^{2}_\mathrm{avg}(3\tilde{\theta})}^{p-\frac{2p}{r}} \\ + \rap(R)\,\|f\|_2^p.
\end{multline}
Each cap in $\Theta_{\Kc^{1/2}R^{-1/2}}$ is covered by
$R^{-1/2}$-caps with bounded overlap, so the average over the larger cap is bounded by a constant multiple of the
maximum of the averages over the $R^{-1/2}$-caps. By Lemma~\ref{norm_l2_local},
\begin{equation} \label{hapiha}
    \max_{\theta \in \Theta_{ R^{-1/2}}}\|f_k\|_{L^{2}_\mathrm{avg}(3\tilde{\theta})}^{p-\frac{2p}{r}} \lesssim \max_{\theta \in \Theta_{ R^{-1/2}}}\|f\|_{L^{2}_\mathrm{avg}(3\tilde{\theta})}^{p-\frac{2p}{r}} + \rap(R)\,\|f\|_2^{p-\frac{2p}{r}}.
\end{equation}
Combining these two facts with \eqref{one-end-case-1}, we obtain
\[\|T^\la f_k\|_{L^p(B_k)}^p \lesssim C_\e \Bigl(\frac{R}{\Kc}\Bigr)^{p\e}\|f_k\|_2^\frac{2p}{r}\,\max_{\theta \in \Theta_{ R^{-1/2}}}\|f\|_{L^{2}_\mathrm{avg}(3\tilde{\theta})}^{p-\frac{2p}{r}} + \rap(R)\,\|f\|_2^p,\]
and then, by \eqref{cp1},
\begin{multline} \label{one-end-case-2}
    \|T^\la f\|_{L^p(B_R^n(0))}^p \lesssim C_\e \Kc^{-p\e}(\log R)^{O_\varepsilon(1)} R^{p\e}\sum_{B_k}\|f_k\|_2^\frac{2p}{r}\,\max_{\theta \in \Theta_{ R^{-1/2}}}\|f\|_{L^{2}_\mathrm{avg}(3\tilde{\theta})}^{p-\frac{2p}{r}} \\ + \rap(R)\,\|f\|_2^p.
\end{multline}
Each plank meets only $O(1)$ of the cubes $B_k$ in a given horizontal region. Since $\beta\le 100$, each wave packet $f_T$ therefore occurs in only $O(1)$ of the functions $f_k$. Moreover, $r \leq p$ by \eqref{prsetting}. Consequently, Lemma \ref{l2-orthogonalty} yields
\[\sum_k\|f_k\|_2^\frac{2p}{r}\lesssim \|f\|_2^\frac{2p}{r}.\]
Inserting this into \eqref{one-end-case-2}, we get
\[
\begin{gathered}
\|T^\la f\|_{L^p(B_R^n(0))}^p \lesssim C_\e \Kc^{-p\e}(\log R)^{O_\varepsilon(1)} R^{p\e}\|f\|_2^\frac{2p}{r}\,\max_{\theta \in \Theta_{ R^{-1/2}}}\|f\|_{L^{2}_\mathrm{avg}(3\tilde{\theta})}^{p-\frac{2p}{r}} \\ + \rap(R)\,\|f\|_2^p.
\end{gathered}
\]
The factor $\Kc^{-p\e}(\log R)^{O_\varepsilon(1)}$ is as small as we wish when $R$ is large. Hence \eqref{eq:induction-claim} holds at scale $R$, and the induction closes.

\subsection{The two-ends case: reduction to a broad estimate}\label{subsec:two-ends-narrow}
Suppose now that $\be > 100$. There are $\Kc^{O(1)}$ cubes $B_k$, so by pigeonholing we can find a cube $B_k$ such that
\begin{equation}
\label{two1}
    \|T^\la f\|_{L^p(B_R^n(0))}^p\lessapprox \Kc^{O(1)}\sum_{B\in\cB, B\subset B_k}\|T^\la f_k\|_{L^p(B)}^p + \rap(R)\,\|f\|_2^p.
\end{equation}
We fix this cube, and we let
\[
A:=2A_*,\qquad X:=\bigcup_{B\in\cB,\,B\subset B_k}B,
\]
where $A_*$ is the number chosen in
Section~\ref{section-algorithm}. Let $B=B_{K^2}^n(y)\subset X$ be a ball of $\cB$, let $v_1,\ldots,v_A\in S^{n-1}$
attain the minimum in the definition of the broad norm on $B$ (see Section~\ref{sectionbroadnorm}), and write
$f_{k,\tau}:=(f_k)_\tau$ for $\tau\in\Theta_{K^{-1}}$.
For each $v_i$, only $O(1)$ of the $K^{-1}$-caps $\tau$ have Gauss directions at $y$ within angle $O(K^{-1})$ of $v_i$, so
$O(A)$ caps are excluded from the maximum in that definition. Applying H\"older's
inequality separately to the sum over these caps and to the sum over the remaining caps, we get
\[
 \|T^\lambda f_k\|_{L^p(B)}^p
 \lesssim A^{p-1}
 \sum_{\tau\in\Theta_{K^{-1}}}
 \|T^\lambda f_{k,\tau}\|_{L^p(B)}^p+K^{O(1)}\|T^\lambda f_k\|_{\BL^p_A(B)}^p.
\]
The factor $A^{p-1}$ is bounded by a power of $\log R$. Summing over the balls $B\subset X$ and using \eqref{two1}, we obtain
\begin{multline} \label{broadreductionn}
\|T^\la f\|_{L^p(B_R^n(0))}^p \lessapprox \Kc^{O(1)} \sum_{\tau\in \Theta_{K^{-1}}} \|T^\la f_{k,\tau}\|_{L^p(B_R^n(0))}^p \\ + \big(\Kc K\big)^{O(1)}\|T^\la f_k\|_{\BL^p_A(X)}^p + \rap(R)\,\|f\|_2^p.
\end{multline}
The first term on the right-hand side is the narrow part. We apply the induction hypothesis at scale $R/K^2$, Lemma \ref{rescaling}, and \eqref{hapiha} to each of the functions $f_{k,\tau}$. The reduced scale $R/K^2$ gives a factor $K^{-2p\varepsilon}$, of which we keep $K^{-p\varepsilon}$ and use the rest to absorb the constants that arise from the normalization. This yields
\[
\begin{gathered}
\|T^\la f_{k,\tau}\|_{L^p(B_R^n(0))}^p \lesssim_\phi C_\e K^{n-(n-2)(p-\frac pr)-p\varepsilon} R^{p\varepsilon}\,
\|f_{k,\tau}\|_2^\frac{2p}{r}
\max_{\theta \in \Theta_{R^{-1/2}}}
\|f\|_{L^2_{\mathrm{avg}}(3\tilde{\theta})}^{p-\frac{2p}{r}} \\ + \rap(R)\,\|f\|_2^p.
\end{gathered}
\]
Since $r\leq p$ and $\sum_\tau\|f_{k,\tau}\|_2^2\lesssim\|f\|_2^2$, inserting this into \eqref{broadreductionn} gives
\[
\begin{gathered}
\|T^\la f\|_{L^p(B_R^n(0))}^p \lessapprox_\phi C_\e\,\Kc^{O(1)} K^{n-(n-2)(p-\frac pr)-p\varepsilon}R^{p\varepsilon}\,
\|f\|_2^\frac{2p}{r}
\max_{\theta \in \Theta_{R^{-1/2}}}
\|f\|_{L^2_{\mathrm{avg}}(3\tilde{\theta})}^{p-\frac{2p}{r}} \\ + \big(\Kc K\big)^{O(1)}\|T^\la f_k\|_{\BL^p_A(X)}^p + \rap(R)\,\|f\|_2^p.
\end{gathered}
\]
By \eqref{prsetting}, the main power of $K$ produced by the Lorentz rescaling vanishes:
\[
 n-(n-2)\left(p-\frac pr\right)=0.
\]
The remaining factor $\Kc^{O(1)}K^{-p\varepsilon}$, together with the
logarithmic losses, is small when $R$ is large, because $\Kc\ll K$.
Thus the narrow part is bounded by a small multiple of the right-hand side of \eqref{eq:induction-claim}, and it remains
to prove the broad estimate
\begin{equation}\label{broadcasereduction}
    \|T^\la f_k\|_{\BL^p_A(X)} \lesssim_{\phi,\varepsilon} R^{\varepsilon^2} \|f\|_2^\frac2r \max_{\theta \in \Theta_{R^{-1/2}}}
\|f\|_{L^2_{\mathrm{avg}}(3\tilde{\theta})}^{1-\frac2r}.
\end{equation}
The loss $R^{\varepsilon^2}$ is small enough to absorb the factor $(\Kc K)^{O(1)}$ in \eqref{broadreductionn}, so that \eqref{broadcasereduction} implies \eqref{eq:induction-claim} at scale $R$ if the constant $C_\e$ is large enough.

\subsection{The two-ends case: the broad estimate}
We keep the cube $B_k$, the function $f_k$, the set $X$, and the number
$A=2A_*$ from the preceding subsection. All the wave packets of $f_k$ belong
to the family $\mathbb T$, so they satisfy $\|f_T\|_2\sim u$, and
\eqref{eq:uniform-packet-energy} and
\eqref{eq:uniform-cap-lower-bound} are available.
To simplify the formulas, we omit the rapidly decaying errors, which are bounded by $\rap(R)\|f\|_2$, until the final estimates.

\subsubsection{Refined decoupling}
Set
\[
 q=\frac{2n}{n-2}.
\]
Let $B\subset X$ be a ball of $\cB$, and let $\tau_1,\ldots,\tau_{A_0}\in\Theta_{K^{-1}}$ attain the minimum in the definition of $\#_{\mathrm{br},A_0}\mathbb T^{\mathrm{prior}}(B;\mathcal G^{\mathrm{prior}})$. By \eqref{eq:stable-ball-count},
fewer than $2\mu$ planks $T\in\mathbb T^{\mathrm{prior}}(B;\mathcal G^{\mathrm{prior}})$ satisfy $\theta(T)\not\subset\bigcup_j4\tau_j$. Every plank of $\mathbb T_k$ that meets $B$ belongs to $\mathbb T(B;\cG)$, and hence to $\mathbb T^{\mathrm{prior}}(B;\mathcal G^{\mathrm{prior}})$.
In the broad norm on $B$ we exclude all the $K^{-1}$-caps whose cutoff functions meet $\bigcup_j4\tau_j$, which is possible by the choice of $C_*$. For each of the remaining $K^{-1}$-caps $\tau$, only $O(\mu)$ wave packets of $f_{k,\tau}$ pass through $B$.

Write $P_\tau f=\tilde\psi_\tau f$ for the cutoff to the $K^{-1}$-cap $\tau$, so that $f_{k,\tau}=P_\tau f_k$. There are $K^{O(1)}$ caps in $\Theta_{K^{-1}}$. By pigeonholing, at the cost of a factor $K^{O(1)}$ we may replace the broad norm of $T^\lambda f_k$ over $X$ by the $L^q$ norm of a single function $T^\lambda f_{k,\tau}$ over a union $Y\subset X$ of balls of $\cB$, each of which meets $O(\mu)$ of the planks $T\in\mathbb T_k$ with $P_\tau f_T\neq0$. The functions $P_\tau f_T$ have the same properties as the wave packets $f_T$, up to fixed enlargements of the $R^{-1/2}$-caps and of the planks, and Corollary~\ref{essential constant decoupling} gives
\[
 \|P_\tau f_T\|_2\lesssim u,
 \qquad
 \|T^\lambda(P_\tau f_T)\|_{L^q(w_{B_R})}\lesssim_\phi u.
\]
After a further pigeonholing, which makes the norms $\|T^\lambda(P_\tau f_T)\|_{L^q(w_{B_R})}$ comparable to each other, we apply
Theorem~\ref{refdecthm} with $M\lesssim\mu$. Since $\frac q2-1=\frac{2}{n-2}$ and at most $\#\mathbb T$ wave packets are involved, this gives
\begin{equation}\label{eq:refined-common-count}
 \|T^\lambda f_k\|_{\BL^q_{A_*}(X)}
 \lessapprox_\phi K^{O(1)}
 \mu^{1/n}(\#\mathbb T)^{1/q}u.
\end{equation}

\subsubsection{\texorpdfstring{$L^2$}{L2} estimates}
In this part we assume that $n\geq4$. The main ingredient is the transverse covering argument of \cite[Section~3]{WangDensity}, which applies in the present setting. By \eqref{eq:stable-ball-count}, each ball of $X$ meets at least $\mu$ planks of $\mathbb T^{\mathrm{prior}}$ whose directions are separated from any given $K^{-1}$-cap, and by \eqref{eq:stable-prior-count}, only $O(\kappa m)$ planks are needed for this in each cube $Q\in\mathcal Q$. Covering $X$ by these planks, and using that each plank of $\mathbb T$ meets at most $l$ cubes of $\mathcal Q$, one obtains the bound $K^{O(1)}(lm/\mu)^{1/2}\|f\|_2$ for the broad $L^2$ norm of $T^\lambda f_k$ over $X$. Together with the first inequality in \eqref{eq:hairbrush-cancellation}, this gives the following estimate.
\begin{lemma}[Transverse covering $L^2$ estimate, \cite{WangDensity}] \label{lem:covering-l2-count}
We have
\[
 \|T^\lambda f_k\|_{\BL^2_{A_*}(X)}
 \lessapprox_\phi K^{O(1)}
 (\#\mathbb T)^{1/2}\mu^{-1}\|f\|_2.
\]
\end{lemma}
We also have the standard $L^2$ estimate. By the fixed-time $L^2$ bound of Lemma~\ref{essential constant}, the
bounded overlap of the cutoff functions of the $K^{-1}$-caps, and the almost
orthogonality of the wave packets,
\[
 \|T^\lambda f_k\|_{\BL^2_{A_*}(X)}^2
 \leq\sum_\tau
       \|T^\lambda P_\tau f_k\|_{L^2(B_{2R})}^2\lesssim_\phi R\sum_\tau\|P_\tau f_k\|_2^2
 \lesssim R\|f_k\|_2^2
 \lesssim R\|f\|_2^2.
\]
Taking the geometric mean of this bound and the bound of
Lemma~\ref{lem:covering-l2-count}, with the weights $(n-2)/n$ and $2/n$ respectively
at the level of the norms, we obtain
\begin{equation}\label{eq:mixed-l2}
 \|T^\lambda f_k\|_{\BL^2_{A_*}(X)}
 \lessapprox_\phi K^{O(1)}
 R^{\frac{n-2}{2n}}(\#\mathbb T)^{1/n}
 \mu^{-2/n}\|f\|_2.
\end{equation}

\subsubsection{Interpolation for \texorpdfstring{$n\geq4$}{n>=4}}
We now express the number $u$ in \eqref{eq:refined-common-count} in terms of the norms of $f$. Take $\gamma=2/n$ in \eqref{eq:uniform-cap-lower-bound} and write
\[
 u=u^{\frac{n-3}{n-2}}u^{\frac1{n-2}}.
\]
We bound the first factor by means of
\eqref{eq:uniform-packet-energy} and the second by means of
\eqref{eq:uniform-cap-lower-bound}. Then \eqref{eq:refined-common-count} becomes
\begin{multline}\label{eq:refined-q-new-endpoint}
 \|T^\lambda f_k\|_{\BL^q_{A_*}(X)}
 \lessapprox_\phi K^{O(1)}
 R^{-\frac{n-2}{4n}}\mu^{1/n}(\#\mathbb T)^{-1/(2n)}\\
 \times\|f\|_2^{\frac{n-3}{n-2}}
 \max_{\theta\in\Theta_{R^{-1/2}}}
 \|f\|_{L^2_{\mathrm{avg}}(3\widetilde\theta)}^{\frac1{n-2}}.
\end{multline}
Since
\[
 \frac1{p_n}=\frac13\frac12+\frac23\frac1q,
 \qquad
 \frac2{r_n}=\frac13+\frac23\frac{n-3}{n-2},
\]
the H\"older inequality for broad norms, Lemma~\ref{lemholder}, gives
\[
 \|T^\lambda f_k\|_{\BL^{p_n}_{2A_*}(X)}
 \lesssim
 \|T^\lambda f_k\|_{\BL^2_{A_*}(X)}^{1/3}
 \|T^\lambda f_k\|_{\BL^q_{A_*}(X)}^{2/3}.
\]
We insert \eqref{eq:mixed-l2} and
\eqref{eq:refined-q-new-endpoint}. The powers of
$R$, $\#\mathbb T$, and $\mu$ cancel, because
\[
 \frac13\frac{n-2}{2n}-\frac23\frac{n-2}{4n}=0,\quad
 \frac13\frac1n-\frac23\frac1{2n}=0,\quad
 -\frac13\frac2n+\frac23\frac1n=0.
\]
Consequently,
\[
 \|T^\lambda f_k\|_{\BL^{p_n}_{2A_*}(X)}
 \lessapprox_\phi K^{O(1)}\|f\|_2^{2/r_n}
 \max_{\theta\in\Theta_{R^{-1/2}}}
 \|f\|_{L^2_{\mathrm{avg}}(3\widetilde\theta)}^{1-2/r_n}+\rap(R)\|f\|_2.
\]
The losses on the right-hand side are bounded by $R^{\varepsilon^2}$, and this proves \eqref{broadcasereduction} for
$n\geq4$.

\subsubsection{The case \texorpdfstring{$n=3$}{n=3}}
When $n=3$, we have $p_3=r_3=4$ and $q=6$, and a simpler $L^2$ estimate suffices.
On each cube $Q\in\mathcal Q$, the function $T^\lambda f_k$ agrees, up to a rapidly decaying error, with the contribution of the wave packets whose planks meet
$Q$. We apply Corollary~\ref{l2 estimate} to these wave packets and sum over the $K^{-1}$-caps, using the bounded overlap of their cutoff functions and the almost orthogonality of the wave packets.
Since each plank meets at most $l$ cubes of $\mathcal Q$, we obtain
\begin{equation}\label{l2-final}
 \|T^\lambda f_k\|_{\BL^2_{A_*}(X)}^2
 \lessapprox_\phi K^{O(1)}R^{1/2}
 \sum_{Q\in\mathcal Q}
 \sum_{\substack{T\in\mathbb T_k\\T\cap Q\ne\varnothing}}
 \|f_T\|_2^2 \lesssim K^{O(1)}R^{1/2}l\|f\|_2^2.
\end{equation}
Next, take $\gamma=1/2$ in \eqref{eq:uniform-cap-lower-bound} and write
$u=u^{1/3}u^{2/3}$ in \eqref{eq:refined-common-count}.
Bounding the first factor by means of \eqref{eq:uniform-packet-energy} and the second by means of \eqref{eq:uniform-cap-lower-bound}, we get
\[
 \|T^\lambda f_k\|_{\BL^6_{A_*}(X)}
 \lessapprox_\phi K^{O(1)}R^{-1/12}
 \mu^{1/3}(\#\mathbb T)^{-1/6}\|f\|_2^{1/3}\cdot\max_{\theta\in\Theta_{R^{-1/2}}}
 \|f\|_{L^2_{\mathrm{avg}}(3\widetilde\theta)}^{2/3}.
\]
The H\"older inequality for broad norms with the weights $1/4$ and $3/4$, together with
\eqref{l2-final} and the second inequality in \eqref{eq:hairbrush-cancellation}, now yields
\[
 \begin{aligned}
 \|T^\lambda f_k\|_{\BL^4_{2A_*}(X)}
 &\lesssim
 \|T^\lambda f_k\|_{\BL^2_{A_*}(X)}^{1/4}
 \|T^\lambda f_k\|_{\BL^6_{A_*}(X)}^{3/4}\\
 &\lessapprox_\phi K^{O(1)}
 \left(\frac{l\mu^2}{\#\mathbb T}\right)^{1/8}\|f\|_2^{1/2}
 \max_{\theta\in\Theta_{R^{-1/2}}}
 \|f\|_{L^2_{\mathrm{avg}}(3\widetilde\theta)}^{1/2}\\
 &\lessapprox_\phi K^{O(1)}\|f\|_2^{1/2}
 \max_{\theta\in\Theta_{R^{-1/2}}}
 \|f\|_{L^2_{\mathrm{avg}}(3\widetilde\theta)}^{1/2}.
 \end{aligned}
\]
Restoring the rapidly decaying error and bounding the losses by $R^{\varepsilon^2}$, we obtain \eqref{broadcasereduction} for $n=3$. This completes the proof of Proposition~\ref{redution 3}.

\section{An \texorpdfstring{$\varepsilon$}{epsilon}-removal lemma}\label{epsilon removal}
In this section we show that Proposition \ref{epsilon main theorem} implies Theorem \ref{main theorem}. Away from the endpoint, the factor $R^{\varepsilon}$ in the local estimates of Proposition \ref{epsilon main theorem} can be removed by means of an $\varepsilon$-removal lemma of the type introduced in \cite{T}. The form of the lemma that we need does not seem to appear in the literature, but it follows from a minor adaptation of the arguments in \cite{T} and \cite{GHI}. In this section we denote the frequency variable by $\omega$, as in \cite{GHI}.

Let $2 \leq r \leq p_0$, and let $T^{\lambda}$ be an oscillatory integral operator satisfying the cinematic curvature condition. Suppose that, for every $\varepsilon>0$, the estimate
\begin{equation} \label{assumpiton 4}
\|T^{\lambda}f\|_{L^p(B^n_R)} \lesssim_{\phi,a,\varepsilon,p} R^\varepsilon \|f\|_{L^r(\mathbb{R}^{n-1})}
\end{equation}
holds uniformly for $1 \leq R \leq \lambda$ whenever $p \geq p_0$. We shall show that the global estimate
\begin{equation} \label{globalestimate}
    \|T^{\lambda}f\|_{L^p(\mathbb{R}^n)} \lesssim_{\phi,a,p} \|f\|_{L^r(\mathbb{R}^{n-1})}
\end{equation}
then holds for all $p > p_0$.

The first step is a reduction of \eqref{globalestimate} to estimates over sparse
collections of balls of radius $R$.

\begin{definition}[\cite{T}] Let $R \geq 1$. A collection $\{B(x_j, R)\}_{j=1}^N$ of $R$-balls in $\R^n$ is \emph{sparse} if the centers $\{x_1, \dots, x_N\}$ are $(RN)^{\bar{C}}$-separated. Here $\bar{C} \geq 1$ is a sufficiently large absolute constant. \end{definition}

\begin{lemma}[{\cite[Lemma 12.2]{GHI}}]\label{redution 5} Let $p_0 \geq 2$, and let $T^{\lambda}$ be an oscillatory integral operator satisfying the cinematic curvature condition. Assume that \eqref{assumpiton 4} holds uniformly for $1 \leq R \leq \lambda$ whenever $p \geq p_0$. To prove \eqref{globalestimate} for all $p > p_0$, it suffices to show that, for all $\varepsilon >0$, the estimate
\[
\|T^{\lambda}f\|_{L^{p_0}(S)} \lesssim_{\varepsilon, \phi, a} R^{\varepsilon}\|f\|_{L^r(\mathbb{R}^{n-1})}
\] holds whenever $R \geq 1$ and $S \subseteq \R^n$ is a union of $R$-balls belonging to a sparse collection. \end{lemma}
The main result of this section is the following proposition, which is proved by a minor modification of the argument in \cite[Lemma 12.3]{GHI}.
\begin{proposition}\label{prop:sparse}
    Let $p_0 \geq 2$, and let $T^{\lambda}$ be an oscillatory integral operator satisfying the cinematic curvature condition. Assume that \eqref{assumpiton 4} holds uniformly for $1 \leq R \leq \lambda$ whenever $p \geq p_0$. Then, for all $\varepsilon >0$, the estimate
\[
\|T^{\lambda}f\|_{L^{p_0}(S)} \lesssim_{\varepsilon, \phi, a} R^{\varepsilon}\|f\|_{L^r(\mathbb{R}^{n-1})}
\] holds whenever $R \geq 1$ and $S \subseteq \R^n$ is a union of $R$-balls belonging to a sparse collection.
\end{proposition}
\begin{proof}
Let $\{B(x_j,R)\}_{j=1}^N$ be a sparse collection of balls whose union is
the set $S$. We may assume that $R \ll \lambda$ and that each
ball $B(x_j,R)$ meets the $x$-support of $a^{\lambda}$. Let
$c_{\mathrm{diam}} > 0$ and $c_{\mathrm{crit}}>0$ be sufficiently small constants, chosen to meet the
requirements of the argument below. By means of a partition of unity, we may
further assume that the $x$-support of $a$ has diameter less than $c_{\mathrm{diam}}$, and hence
\[
\frac{|x_{j_1} - x_{j_2}|}{\lambda} \lesssim c_{\mathrm{diam}} \qquad \textrm{for all $1 \leq j_1, j_2 \leq N$.}
\]

Fix a function $\eta \in C^{\infty}(\mathbb{R}^{n-1})$ such that
$0 \le \eta \le 1$, $\supp \eta \subset B^{n-1}(0,1)$, and
$\eta(z)=1$ for all $z \in B^{n-1}(0,1/2)$. Let
$R_1 := CR$, where $C \ge 1$ is a sufficiently large constant, and
define $\eta_{R_1}(z) := \eta(z/R_1)$. Let $R_2 := (RN)^{\bar{C}}$.
In addition, let $\mathcal W\subset\R^{n-1}$ be a small open neighborhood of the $\omega$-support of $a$, and let $\psi \in C_c^{\infty}(\mathbb{R}^{n-1})$ satisfy
$0 \le \psi \le 1$, $\supp \psi \subset \mathcal W$, and
$\psi(\omega)=1$ for all $\omega$ in the $\omega$-support of $a$.

Fix $1 \le j \le N$. We decompose
\[
e^{2\pi i \phi^{\lambda}(x_j;\,\cdot\,)} \psi f
= P_j f
+ \bigl(e^{2\pi i \phi^{\lambda}(x_j;\,\cdot\,)} \psi f - P_j f\bigr)
=: P_j f + f_{j,\infty},
\]
where
\[
P_j f := \widehat{\eta}_{R_1} *
\bigl[e^{2\pi i \phi^{\lambda}(x_j;\,\cdot\,)} \psi f\bigr].
\]
As shown in \cite[Proof of Lemma 12.3]{GHI}, it suffices to prove that
\begin{equation} \label{77491}
\big(\sum_{j=1}^N \|P_j f\|_{L^r(\mathbb{R}^{n-1})}^p\big)^\frac{1}{p} \lesssim \|f\|_{L^r(\mathbb{R}^{n-1})}.
\end{equation}
We obtain this estimate by interpolating between the cases $p=r$ and $p=\infty$. By Young's inequality,
\[\|P_j f\|_{L^r(\mathbb{R}^{n-1})} \leq \|\widehat{\eta}_{R_1}\|_{L^1(\mathbb{R}^{n-1})} \|f\|_{L^r(\mathbb{R}^{n-1})} \lesssim \|f\|_{L^r(\mathbb{R}^{n-1})},\]
so \eqref{77491} holds for $p = \infty$. It remains to show that
\[\big(\sum_{j=1}^N \|P_j f\|_{L^r(\mathbb{R}^{n-1})}^r\big)^\frac{1}{r} \lesssim \|f\|_{L^r(\mathbb{R}^{n-1})}.\]
This estimate follows in turn by interpolation between the cases $r=2$ and $r=\infty$. For $r=\infty$ it is again a consequence of Young's inequality. For $r=2$ it suffices, by duality, to show that
\begin{equation} \label{77492}
\|\sum_{j=1}^N P_j^* g_j\|_{L^2(\R^{n-1})}^2 \lesssim \sum_{j=1}^N \| g_j\|_{L^2(\R^{n-1})}^2
\end{equation}
uniformly for all $g_1,g_2,\dots,g_N \in L^2(\mathbb{R}^{n-1})$. The left-hand side of \eqref{77492} equals
\begin{equation} \label{expressionn}
\sum_{j_1, j_2 = 1}^{N} \int_{\R^{n-1}} \overline{G_{j_1,j_2}(\omega)} \hat{\eta}_{R_1} \ast g_{j_1}(\omega)\overline{\hat{\eta}_{R_1} \ast g_{j_2}(\omega)}\ d\omega,
\end{equation}
where
\[
G_{j_1,j_2}(\omega) := e^{2\pi i (\phi^{\lambda}(x_{j_1}; \omega) - \phi^{\lambda}(x_{j_2}; \omega))} \psi(\omega)^2.
\]
By Plancherel's theorem, each summand in \eqref{expressionn} can be written as
\[
\int_{\R^{n-1}} \overline{\check{G}_{j_1,j_2}(z)} (\eta_{R_1}\check{g}_{j_1}) \ast (\eta_{R_1}\check{g}_{j_2})^{\sim}(z) d z;
\]
here $(\eta_{R_1}\check{g}_{j_2})^{\sim}(z) := \overline{(\eta_{R_1}\check{g}_{j_2})(-z)}$.

Fix $1 \leq j_1, j_2 \leq N$ with $j_1 \neq j_2$, and let $z \in \R^{n-1}$ satisfy $|z| \lesssim R_1 \ll R_2 \leq |x_{j_2} - x_{j_1}|$. Consider
\[ \check{G}_{j_1,j_2}(z) = \int_{\R^{n-1}} e^{2\pi i (\langle z, \omega \rangle + \phi^{\lambda}(x_{j_1}; \omega) - \phi^{\lambda}(x_{j_2}; \omega))} \psi(\omega)^2 d\omega.\]
We shall cover $\mathcal W$ by $O_\phi(1)$ open sets
$D$ such that
\begin{equation}\label{bound each off diagonal}
\int_{\mathbb{R}^{n-1}}
e^{2\pi i \bigl(\langle z,\omega\rangle
+ \phi^{\lambda}(x_{j_1};\omega)
- \phi^{\lambda}(x_{j_2};\omega)\bigr)}
\,\psi(\omega)^2\,\psi_D(\omega) d\omega
\;\lesssim\;
R_2^{-\frac{n-2}{2}},
\end{equation}
where $\{\psi_D\}$ is a smooth partition of unity subordinate to this cover. Once this is done, \eqref{77492} follows from the Hausdorff--Young inequality and H\"older's inequality.

To define the cover, let $C'>0$ be a large constant, and let
\[ D_0:= \{\om: \om \in \mathcal W, \big| \pm \frac{x_{j_2} - x_{j_1}}{|x_{j_2} - x_{j_1}|} - G^{\lambda}(x_{j_1};\omega)\big| > C'c_{\mathrm{crit}} \}, \]
\[ U_0:= \{\om: \om \in \mathcal W, \big| \pm \frac{x_{j_2} - x_{j_1}}{|x_{j_2} - x_{j_1}|} - G^{\lambda}(x_{j_1};\omega)\big| < 2C'c_{\mathrm{crit}} \}. \]
Define
\[M(\om_0) := \partial_{\omega \omega}\langle \partial_x\phi^{\lambda}(x_{j_1}; \omega), G^{\lambda}(x_{j_1};\omega_0) \rangle|_{\omega = \omega_0},\]
and let $M_i(\om_0)$ denote the matrix obtained from $M(\om_0)$ by deleting the row and the column that correspond to $\om_i$. Let $c'>0$ be a sufficiently small constant, and define
\begin{equation} \label{finaldef}
U_i := \{\om: \om \in \mathcal W, \det M_i(\om) > c' \}, \qquad D_i := U_i \cap U_0.
\end{equation}
Since $M(\omega)$ has rank $n-2$ and $n-2$ positive eigenvalues, the sets $D_0,D_1,\dots,D_{n-1}$ form an open cover of $\mathcal W$. We choose a smooth partition of unity $\{\psi_{D_i}\}_{i=0}^{n-1}$ subordinate to this cover.

We now verify \eqref{bound each off diagonal}. If $\om \in D_0$, then $x_{j_1}-x_{j_2}$ makes a quantitative angle with the direction $G^\lambda(x_{j_1};\omega)$, which spans the kernel of the matrix $\partial_{\omega x}\phi^{\lambda}(x_{j_1};\omega)$. Hence
\[
\begin{gathered}
\partial_{\omega} \bigl(\langle z,\omega\rangle + \phi^{\lambda}(x_{j_1};\omega) - \phi^{\lambda}(x_{j_2};\omega)\bigr) = \langle \partial_{\omega x}\phi^\la(x_{j_1}; \omega), x_{j_1} - x_{j_2} \rangle \\ + O(c_{\mathrm{diam}}|x_{j_1} - x_{j_2}|) \gtrsim |x_{j_1} - x_{j_2}|,
\end{gathered}
\]
and repeated integration by parts gives
\[
\begin{gathered}
\int_{\mathbb{R}^{n-1}}
e^{2\pi i \bigl(\langle z,\omega\rangle
+ \phi^{\lambda}(x_{j_1};\omega)
- \phi^{\lambda}(x_{j_2};\omega)\bigr)}
\,\psi(\omega)^2\,\psi_{D_0}(\omega) d\omega \\
\lesssim |x_{j_1}-x_{j_2}|^{-(n-1)} \lesssim R_2^{-(n-1)}.
\end{gathered}
\]
If $\om \in D_i$ for some $1\leq i \leq n-1$, then
\[\partial_{\omega' \omega'} \bigl(\langle z,\omega\rangle + \phi^{\lambda}(x_{j_1};\omega) - \phi^{\lambda}(x_{j_2};\omega)\bigr) = |x_{j_1} - x_{j_2}| M_i(\om) + O(C'c_{\mathrm{crit}}|x_{j_1} - x_{j_2}|), \]
where $\om' = (\om_1,\dots, \om_{i-1},\om_{i+1},\dots,\om_{n-1})$. By \eqref{finaldef}, if $C'c_{\mathrm{crit}}$ is small enough, then
\[ \det \partial_{\omega' \omega'} \bigl(\langle z,\omega\rangle + \phi^{\lambda}(x_{j_1};\omega) - \phi^{\lambda}(x_{j_2};\omega)\bigr) \gtrsim |x_{j_1} - x_{j_2}|^{n-2}.\]
Applying the higher-dimensional van der Corput lemma in the variables $\omega'$, and then integrating over the bounded $\omega_i$-support of the cutoff, we obtain
\[
\begin{gathered}
\int_{\mathbb{R}^{n-1}}
e^{2\pi i \bigl(\langle z,\omega\rangle
+ \phi^{\lambda}(x_{j_1};\omega)
- \phi^{\lambda}(x_{j_2};\omega)\bigr)}
\,\psi(\omega)^2\,\psi_{D_i}(\omega) d\omega \\ = \int_{\mathbb{R}} \int_{\mathbb{R}^{n-2}}
e^{2\pi i \bigl(\langle z,\omega\rangle
+ \phi^{\lambda}(x_{j_1};\omega)
- \phi^{\lambda}(x_{j_2};\omega)\bigr)}
\,\psi(\omega)^2\,\psi_{D_i}(\omega) d\omega' d\omega_i \\ \lesssim |x_{j_1} - x_{j_2}|^{-\frac{n-2}{2}} \lesssim R_2^{-\frac{n-2}{2}}.
\end{gathered}
\]
This proves \eqref{bound each off diagonal} in all cases.
\end{proof}
By Proposition~\ref{epsilon main theorem}, Lemma~\ref{redution 5}, and Proposition~\ref{prop:sparse}, the estimate \eqref{linear estimate} holds for all $p$ satisfying \eqref{main range} and all $r$ satisfying \[\frac{r}{r-1} < \frac{n-2}{n} p.\] The boundary case $\frac{r}{r-1} = \frac{n-2}{n} p$ follows from the bilinear interpolation argument in \cite{OW,S}. This completes the proof of Theorem~\ref{main theorem}.


\begin{thebibliography}{99}

\bibitem{BHS}
D. Beltran, J. Hickman, and C. Sogge.
\newblock Variable coefficient Wolff-type inequalities and sharp local smoothing estimates for wave equations on manifolds.
\newblock {\em Anal. PDE} 13 (2020), no. 2, 403--433.

\bibitem{BG}
J. Bourgain and L. Guth.
\newblock Bounds on oscillatory integral operators based on multilinear estimates.
\newblock {\em Geom. Funct. Anal.} 21 (2011), no. 6, 1239--1295.

\bibitem{GW}
S. Gan, D. He, X. Li, and S. Wu.
\newblock On local smoothing estimates for wave equations.
\newblock arXiv:2502.05973v2, 2026.

\bibitem{G}
L. Guth.
\newblock A restriction estimate using polynomial partitioning.
\newblock {\em J. Amer. Math. Soc.} 29 (2016), no. 2, 371--413.

\bibitem{G2}
L. Guth.
\newblock Restriction estimates using polynomial partitioning II.
\newblock {\em Acta Math.} 221 (2018), no. 1, 81--142.

\bibitem{GHI}
L. Guth, J. Hickman, and M. Iliopoulou.
\newblock Sharp estimates for oscillatory integral operators via polynomial partitioning.
\newblock {\em Acta Math.} 223 (2019), no. 2, 251--376.

\bibitem{GIOW}
L. Guth, A. Iosevich, Y. Ou, and H. Wang.
\newblock On Falconer's distance set problem in the plane.
\newblock {\em Invent. Math.} 219 (2020), no. 3, 779--830.

\bibitem{H}
L. H\"ormander.
\newblock Oscillatory integrals and multipliers on $FL^p$.
\newblock {\em Ark. Mat.} 11 (1973), 1--11.

\bibitem{OpenAI2026LocalSmoothing}
OpenAI.
\newblock \href{https://github.com/openai/math/blob/main/preprints/Local-smoothing-for-the-Euclidean-wave-equation-October-25-2026/paper.pdf}{Local smoothing for the Euclidean wave equation}.
\newblock Preprint, October~25, 2026.

\bibitem{OW}
Y. Ou and H. Wang.
\newblock A cone restriction estimate using polynomial partitioning.
\newblock {\em J. Eur. Math. Soc.} 24 (2022), 3557--3595.

\bibitem{S}
R. Schippa.
\newblock Oscillatory integral operators with homogeneous phase functions.
\newblock {\em J. Anal. Math.} 154 (2024), 1--67.

\bibitem{T}
T. Tao.
\newblock The Bochner--Riesz conjecture implies the restriction conjecture.
\newblock {\em Duke Math. J.} 96 (1999), no. 2, 363--375.

\bibitem{WW}
H. Wang and S. Wu.
\newblock Restriction estimates using decoupling theorems and two-ends Furstenberg inequalities.
\newblock arXiv:2411.08871, 2024.

\bibitem{Wang}
X. Wang.
\newblock On cone restriction estimates in higher dimensions.
\newblock {\em J. Fourier Anal. Appl.} 32 (2026), Article 111.

\bibitem{WangDensity}
X. Wang.
\newblock Wave packet density in local smoothing: refinements and limitations.\newline
arXiv:2609.34675, 2026.

\end{thebibliography}
\end{document}